\documentclass[a4paper]{amsart}

\usepackage{etex}
\usepackage[english]{babel}
\usepackage{tikz}
\usetikzlibrary{arrows,decorations.pathmorphing,backgrounds,positioning,fit,petri,patterns}
\usepackage{pgfplots}
\usepackage[utf8]{inputenc}
\usepackage{graphicx}
\usepackage[colorlinks=true,linkcolor=black,citecolor=black]{hyperref}
\usepackage{amsfonts}
\usepackage{enumerate}
\usepackage{booktabs}
\usepackage{array} 
\usepackage{paralist} 
\usepackage{mathtools}
\usepackage{tabu}
\usepackage{amsthm}
\usepackage{amsopn}
\usepackage{wrapfig}
\usepackage[font=small]{caption}
\usepackage{subcaption}
\usepackage{units}
\usepackage[symbol]{footmisc}
\usepackage{todonotes}
\usepackage{amssymb}
\usepackage{pdfpages}
\usepackage{setspace}
\usepackage{enumitem}
\usepackage{bbm}
\usepackage{placeins}

\definecolor{dkgreen}{rgb}{0,0.6,0}
\newcommand{\red}[1]{\textcolor{red}{#1}}

\newcommand{\f}{\frac}

\newcommand{\del}{\partial}

\newcommand{\im}{\operatorname{Im}}

\newcommand{\re}{\operatorname{Re}}

\renewcommand{\epsilon}{\varepsilon}

\newcommand{\dist}{\operatorname{dist}}

\newcommand{\bs}{\backslash}

\newcommand{\R}{\mathbb R}
\newcommand{\C}{\mathbb C}
\newcommand{\N}{\mathbb N}
\newcommand{\Z}{\mathbb Z}

\newcommand{\diam}{\operatorname{diam}}
\renewcommand{\i}{\mathrm{i}}
\newcommand{\diag}{\operatorname{diag}}

\newcommand{\cO}{\mathcal{O}}

\newcommand{\cN}{\mathcal{N}}
\newcommand{\cH}{\mathcal{H}}
\newcommand{\cE}{\mathcal{E}}
\newcommand{\cK}{\mathcal{K}}
\newcommand{\cB}{\mathcal{B}}

\newcommand{\cT}{\mathcal{T}}

\newcommand{\cC}{\mathcal{C}}
\newcommand{\cS}{\mathcal{S}}
\newcommand{\cM}{\mathcal{M}}
\newcommand{\cZ}{\mathcal{Z}}
\newcommand{\cI}{\mathcal{I}}

\renewcommand{\d}{\mathrm{d}}
\newcommand{\wto}{\rightharpoonup}
\newcommand{\Mto}{\xrightarrow{M}}
\newcommand{\Esol}{\cE^n_{\mathrm{sol}}}
\newcommand{\Emat}{\cE^n_{\mathrm{mat}}}
\newcommand{\Esum}{\cE^n_{\mathrm{sum}}}
\newcommand{\Om}{\Omega}

\newcommand{\VhND}{V^n_{\mathrm{ND}}(\Om_n)}
\newcommand{\HND}{H^1_{\mathrm{ND}}(\Om)}

\newcommand{\DND}{\Delta_{\mathrm{ND}}^{\Om}}

\DeclarePairedDelimiter{\abs}{\lvert}{\rvert}
\DeclarePairedDelimiter{\norm}{\lVert}{\rVert}

\DeclarePairedDelimiter{\inner}{\langle}{\rangle}
\DeclarePairedDelimiter{\set}{\lbrace}{\rbrace}
\DeclarePairedDelimiter{\br}{(}{)}

\newcommand{\Omout}{\Om_{\mathrm{out}}}
\newcommand{\Min}{M_{\mathrm{in}}}
\newcommand{\hMinn}{\hat{M}^n_{\mathrm{in}}}
\newcommand{\Mout}{M_{\mathrm{out}}}
\newcommand{\intt}{\mathrm{int}}
\newcommand{\cW}{\mathcal{W}}
\newcommand{\cX}{\mathcal{X}}

\allowdisplaybreaks
\numberwithin{equation}{section}

\theoremstyle{definition}
\newtheorem{de}{Definition}[section]

\newtheorem{example}{Example}
\theoremstyle{plain}
\newtheorem{prop}[de]{Proposition}
\newtheorem{lemma}[de]{Lemma}
\newtheorem{theorem}[de]{Theorem}

\newtheorem{assumption}[de]{Assumption}
\theoremstyle{remark}
\newtheorem{remark}[de]{Remark}

\theoremstyle{plain}
\newtheorem{innercustomgeneric}{\customgenericname}
\providecommand{\customgenericname}{}
\newcommand{\newcustomtheorem}[2]{%
	\newenvironment{#1}[1]
	{%
		\renewcommand\customgenericname{#2}%
		\renewcommand\theinnercustomgeneric{##1}%
		\innercustomgeneric
	}
	{\endinnercustomgeneric}
}

\newcustomtheorem{customthm}{Theorem}
\newcustomtheorem{customlemma}{Lemma}

\AtBeginDocument{

}

\pgfplotsset{compat=newest}

\usepackage{kvsetkeys}
\usepackage{etexcmds}

\makeatletter
\newcount\arg@count
\newcommand{\arg@parser}[1]{%
	\advance\arg@count\@ne
	\expandafter\let\csname arg\romannumeral\arg@count\endcsname\comma@entry
}
\newcommand\res[1]{
	\arg@count=\z@
	\comma@parse{ \lambda,A }\arg@parser 
	\arg@count=\z@
	\comma@parse{#1}\arg@parser
	\ifnum\arg@count>2 %
	\@latex@error{Too many arguments}{%
		The macro \string\mycmd\space got \the\arg@count\space
		arguments,\MessageBreak
		but expected are 2 arguments.\MessageBreak
		\@ehd
	}%
	\fi
	\edef\process@me{%
		\noexpand\@res
		{\etex@unexpanded\expandafter{\argi}}%
		{\etex@unexpanded\expandafter{\argii}}%
	}%
	\process@me
}
\newcommand{\@res}[2]{%
	\ensuremath\left( #1 - #2 \right)^{-1}
}
\makeatother

\makeatletter
\newcount\arg@count
\newcommand\p[1]{
	\arg@count=\z@
	\comma@parse{  }\arg@parser 
	\arg@count=\z@
	\comma@parse{#1}\arg@parser
	\ifnum\arg@count=2 %
	\else
	\@latex@error{Wrong number of mandatory arguments}{%
		The macro \string\p\space got \the\numexpr\arg@count-2\relax\space
		mandatory arguments,\MessageBreak
		but expected are 3 mandatory arguments.\MessageBreak
		\@ehd
	}%
	\fi
	\edef\process@me{%
		\noexpand\@p
		{\etex@unexpanded\expandafter{\argi}}%
		{\etex@unexpanded\expandafter{\argii}}%
	}%
	\process@me
}
\newcommand{\@p}[2]{%
	\ensuremath \left\langle #1 , #2 \right\rangle
}
\makeatother

\thanks{The authors would like to thank Marco Marletta for helpful discussions. }
\title{Computing scattering resonances of rough obstacles}

\author{Frank R\"osler}
\email{frank.roesler@posteo.de}
\address{Department of Mathematics, University of Bern, Alpeneggstrasse 22, 3012 Bern, Switzerland}
\author{Alexei Stepanenko}
\email{as3304@cam.ac.uk}
\address{Department of Applied Mathematics and Theoretical Physics, Centre for Mathematical Sciences, Wilberforce Road, Cambridge CB3 0WA, United Kingdom}

\begin{document}


\maketitle
\begin{abstract}
	This paper is concerned with the numerical computation of scattering resonances of the Laplacian for Dirichlet obstacles with rough boundary. We prove that under mild geometric assumptions on the obstacle there exists an algorithm whose output is guaranteed to converge to the set of resonances of the problem. 
	The result is formulated using the framework of Solvability Complexity Indices. 
The proof is constructive and provides an efficient numerical method. The algorithm is based on a combination of a Glazman decomposition, a polygonal approximation of the obstacle and a finite element method. Our result applies in particular to obstacles with fractal boundary, such as the Koch Snowflake and certain filled Julia sets.
Finally, we provide numerical experiments in MATLAB for a range of interesting obstacle domains. 
	
\end{abstract}

\section{Introduction}

Scattering resonances of the Laplace operator play a fundamental role in quantum mechanics and acoustics, encoding the behaviour of the wave and free Schr\"odinger equations \cite{resonances}. 
The non-self-adjoint nature of the problem poses challenges for the construction of stable and convergent numerical algorithms. Current methods include complex scaling (as well as the closely related method of perfectly matched layers), boundary integral methods and various modifications of the finite element method.  The study of the foundations of computation for scattering resonances was recently initiated in \cite{seashell,JEMS} within the framework of Solvability Complexity Indices \cite{sci1,sci2,sci3,sci4,sci5}.

Meanwhile, in the context of wave propagation and spectral theory, fractal boundaries are responsible for a variety of striking phenomena. Notably, they are particularly efficient for wave absorption \cite{absorb} and  give rise to exotic eigenvalue asymptotics (see \cite{asym1,asym2,asym3,asym4,asym5} for instance and references therein). 
In our previous article \cite{spectralpaper}, we investigated the computability of eigenvalues of the Dirichlet Laplacian on bounded domains.  Our results revealed that eigenvalues are computable in one limit for a wide class of domains with rough boundaries, however, for large enough classes of domains, with sufficiently irregular boundaries, the problem becomes non-computable.

In this article, we study the numerical computation of scattering resonances for the Laplacian on exterior domains in the plane, that is domains of the form $U^c = \R^2 \bs U$ for closed and bounded $U \subset \R^2$, endowed with Dirichlet boundary conditions.  
This setting is often referred to as \emph{obstacle scattering} and $U$ as the \emph{obstacle}.
The novelty of our article is that we only impose very mild geometric conditions on the boundary of the domain, allowing for a wide variety of fractal boundaries.
Our primary contributions are as follows:
\begin{enumerate}
	\item We introduce a fast and simple numerical algorithm for scattering resonances with rough boundaries. 
	\item We prove convergence of this algorithm (see Theorem \ref{th:main}), from which we deduce an upper bound for the computational complexity of the problem (see Theorem \ref{th:SCI}). Our sole geometric assumption on the domain that its boundary is a union of a finite number of closed Jordan curves with zero area. 
	\item  We perform numerical investigations for scattering resonances for Koch snowflake and filled Julia set obstacles. See Section \ref{sec:numerics}. 
\end{enumerate}

Our algorithm is a modification of one introduced by Levitin and Marletta \cite{LevitinMarletta} (see Remark \ref{rem:lev-marl}).
It is based on domain decomposition, Neumann-to-Dirichlet (NtD) operators, spectral expansions and the finite element method (FEM). This gives rise to several sources of approximation error, including:
\begin{itemize}
	\item Geometric error due to the approximation of rough boundaries by polygonal ones.
	\item FEM discretisation error. 
	\item Truncation error due to approximation of NtD operators by finite matrices. 
\end{itemize} 
These sources of error are estimated and linked in order to ensure convergence in one limit (see Assumption \ref{ass:param}). The proof of our main convergence theorem is based on utilising the notion of Mosco convergence in conjunction with Gohberg-Sigal theory.

Our SCI result (Theorem \ref{th:SCI}) generalises the aforementioned work \cite{seashell}. While the result in \cite{seashell} proves existence of a convergent algorithm for scattering resonances of obstacles with $C^2$ boundary, we deal with a vastly wider class of obstacles. 
Other closely related works include similar algorithms in other settings \cite{hyperbolic,roddick}, the boundary element method for fractal screens \cite{CW21} and shape optimisation for rough domains \cite{shape}.

\section{Overview of results}\label{sec:overview}

In this section, we provide the necessary background for scattering resonances, state our main results and provide details of our numerical method.  

\subsection{Scattering resonances}

Let $U \subset \R^2$ denote a closed, bounded set and consider the Laplacian $H$ on the the exterior domain $U^c = \R^2 \bs U$ endowed with homogeneous Dirichlet boundary conditions on $\partial U$,
\begin{equation*}
	H = - \Delta, \qquad D(H) = \set{u \in H^1_0(U^c): \Delta u \in L^2(U^c)}
\end{equation*}
where $D(H)$ denotes the domain of the operator $H$ and we employ standard notation for Sobolev spaces ($H^1$, $H^1_0$ etc.). 
Let $X > 0$ be large enough so that $U \subset B_{X-1}(0)$ (throughout, $B_r(0)$ shall denotes a ball in $\R^2$ of radius $r > 0$ centred at the origin). For simplicity, we restrict our attention to scattering resonances in $\C_-$. 

 Let $\chi$ be a smooth, compactly supported function such that $\chi \equiv 1$ on $B_X(0)$. Then, by \cite[Theorem 4.4]{resonances}, the analytic operator-valued function
\begin{equation}
	\chi (H - k^2)^{-1} \chi:L^2(\R^2) \to L^2(\R^2), \qquad k \in \C_+,
\end{equation}
admits a meromorphic continuation to $k \in \C \bs (-\infty, 0]$.
\begin{de}\label{def:resonance}
	The \emph{scattering resonances} of $H$ are defined as the poles of $\chi (H - k^2)^{-1} \chi$ in $\C_-$. 
\end{de}

\subsection{Main computational complexity result}\label{subsec:comp}

Consider a computational problem described by the following elements. 
\begin{enumerate}[label = (\Alph*)]
\item Let $\cS$ denote the set of closed, bounded sets $U \subset \mathbb{R}^2$ (representing obstacles) such that the following holds. 
\begin{itemize}
	\item $U$ has a finite number of connected components and is the closure of an open set.
	\item Each connected component of $\partial U$ is path-connected and has zero Lebesgue measure (i.e. zero area).
\end{itemize}
$\cS$ is referred to as the \emph{primary set} and represents the class of admissible obstacles. Examples of elements $U \in \cS$ include the Koch snowflake and certain filled Julia sets (see Section \ref{sec:koch_numerics} and \ref{sec:julia_numerics} respectively).
\item   Define a metric space  $\mathcal{M} := (\mathrm{cl}(\C_-),\d_{\mathrm{AW}})$, where $\mathrm{cl}(\C_-)$ denotes the set of closed, non-empty subsets of $\C_-$ and  $\d_{\mathrm{AW}}$ denotes the
\emph{Attouch-Wets distance}, 
\begin{align}\label{eq:att-wet}
	d_{\text{AW}}(A,B) = \sum_{k=1}^\infty 2^{-k}\min\left\{ 1\,,\,\sup_{|x|<k}\left| \dist(x,A) - \dist(x,B) \right| \right\}.
\end{align} 
Note that if $A,B\subset\R^d$ are bounded, then $d_{\text{AW}}$ is equivalent to the Hausdorff distance.
\item Let $\mathrm{Res}: \cS \to \cM$ denote the map which, for any obstacle $U \in \cS$, gives the corresponding set of scattering resonances in $\C_-$. This is referred to as the \emph{problem function}. 
\item Consider a set of real-valued maps defined by 
\begin{equation}
	\Lambda = \set{U \mapsto \mathbbm{1}_U(x) : x \in \R^2},
\end{equation}
where $\mathbbm{1}_U$ denotes the characteristic function of $U$. 
This is referred to as the \emph{evaluation set} and represents the information available to an algorithm. 
\end{enumerate}

Together, the quadruple $\set{\mathrm{Res}, \cS,\cM, \Lambda}$ formally constitutes a computational problem in the language of Solvability Complexity Indices, which may be summarised as:
\begin{equation*}
	\textrm{Compute the scattering resonances for obstacles in }\cS\textrm{, given access point values of }\mathbbm{1}_U\textrm{.}
\end{equation*} 
Intuitively, an \emph{arithmetic algorithm} for this computational problem is a map $\Gamma: \cS \to \cM$ which produces its output by performing a finite number of arithmetic operations on $\mathbbm{1}_U(x_1)$,...,$\mathbbm{1}_U(x_N)$ for some sample points $x_1,...,x_N \in \R^d$, which may depend on $U$. This notion is formalised as follows. 
\begin{de}[Arithmetic algorithm]\label{def:Algorithm}
	Let $\set{\Xi, \cS,\cM, \Lambda}$ be a computational problem. An arithmetic algorithm is a mapping $\Gamma:\cS\to\mathcal M$ such that for each $U\in\cS$ 
	\begin{enumerate}
		\item[(i)] there exists a finite (non-empty) subset $\Lambda_\Gamma(U)\subset\Lambda$,
		\item[(ii)] the action of $\Gamma$ on $U$ depends only on $\{f(U)\}_{f\in\Lambda_\Gamma(U)}$ and the output is obtained by a finite number of arithmetic operations,
		\item[(iii)] for every $U_2 \in\cS$ with $f(U_2)=f(U_1)$ for all $f\in\Lambda_\Gamma(U_1)$ one has $\Lambda_\Gamma(U_1)=\Lambda_\Gamma(U_2)$.
	\end{enumerate}
\end{de}

Our first result reads as follows. 
\begin{theorem}\label{th:SCI}
There exists a sequence of arithmetic algorithms $\Gamma_n: \cS \to \cM$, $n \in \N$, for the computational problem $\set{\mathrm{Res}, \cS,\cM, \Lambda}$ such that 
\begin{equation}\label{eq:SCI-conv}
	\forall\, U \in \cS: \qquad \d_{\mathrm{AW}}(\Gamma_n(U), \mathrm{Res}(U)) \to 0 \qquad \text{as}\qquad n \to \infty.
\end{equation}
\end{theorem}

In the language of SCI, Theorem \ref{th:SCI} may be written 
\begin{equation}
	\set{\mathrm{Res}, \cS,\cM, \Lambda} \in \Delta_2^A,
\end{equation}
where $\Delta_2^A$ refers to the fact that the result states convergence in one limit with arithmetic algorithms.

\subsection{Domain decomposition definition}

Next, we define scattering resonances via an alternative domain decomposition approach, from which our numerical method will naturally follow. 
Decompose $U^c$ into an inner domain, an outer domain and an interface as (see Figure \ref{fig:domain_decomposition})
\begin{equation}
  \label{eq:inner-outer-defn}
  \Om := B_X(0) \bs U,  \qquad \Omout := \R^2 \bs \overline{B_X(0)} \qquad \text{and} \qquad \Gamma := \partial B_X(0).
\end{equation}
Let $\set{e_\alpha}_{\alpha \in \Z}$ denote an orthonormal basis for $L^2(\Gamma)$ defined by
\begin{equation*}
  e_\alpha(\theta) := \frac{1}{\sqrt{2 \pi X}} e^{i \alpha \theta}, \qquad \alpha \in \Z.
\end{equation*}

\begin{figure}[t]
	\centering
	\includegraphics{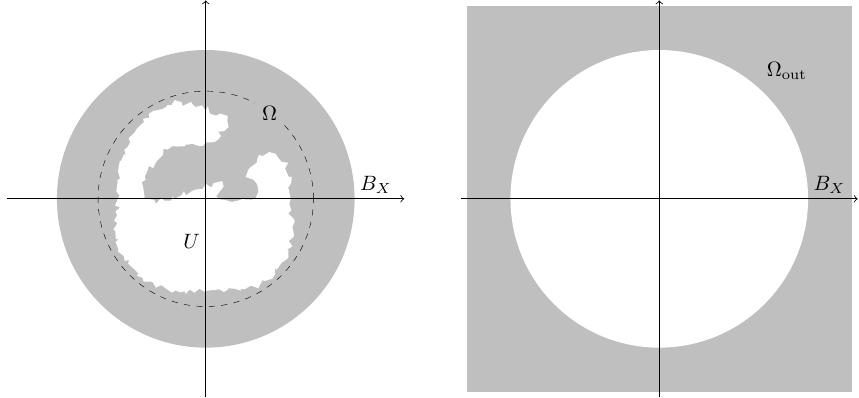}
	\caption{Sketch of the definitions of $\Omega$ and $\Omega_{\text{out}}$. The dashed line in the left figure indicates $B_{X-1}(0)$.}
	\label{fig:domain_decomposition}
\end{figure}

Consider the boundary value problem
\begin{equation}
  \label{eq:BVP-main}
  \begin{cases}
   - \Delta u = k^2 u \quad \text{on} \quad U^c \\
    u|_{\partial U} = 0
  \end{cases}, \qquad k \in \C,
\end{equation}
so that, in particular, $k^2$ is an eigenvalue of $H$ if and only if there exists a non-trivial solution $u$ of \eqref{eq:BVP-main} in $L^2(U^c)$.
Let $H^{(1)}_{n}$ and $H^{(2)}_{n}$ denote the Hankel functions of the first and second kind (of order $n$) respectively. For $k \neq 0$, the general form for solutions of $-\Delta u = k^2 u $ on the outer domain $\Omout$ is
\begin{equation}
  u(x;k) = \sum_{\alpha \in \Z} \br*{q_\alpha H^{(1)}_{|\alpha|}(k |x|) e_{\alpha}(\theta(x)) + \tilde{q}_\alpha H^{(2)}_{|\alpha|}(k |x|) e_{\alpha}(\theta(x)) },
\end{equation}
where $q_\alpha$ and $\tilde{q}_\alpha$ are sequences of complex numbers such that $(q_\alpha H^{(1)}_{|\alpha|}(k |x|))_{\alpha \in \Z}, (\tilde{q}_\alpha H^{(2)}_{|\alpha|}(k |x|))_{\alpha \in \Z} \in \ell^1(\Z)$ and $\theta(x)$ is the angular part of $x$. We select only the ``first kind solutions'' (i.e. the ones with $\tilde{q}_{\alpha} = 0$); these are precisely the set of solutions that are in $L^2(\Omout)$ when $k \in \C_+$ (throughout, $\C_\pm = \set{z \in \C : \pm \im \, z > 0}$).
Furthermore, it is convenient to re-parametrise $q_\alpha$, in order to cancel out the rapid growth of the Hankel functions of first kind for large order,
\begin{equation}\label{eq:Hankel-asym}
H^{(1)}_{\nu}(k X) \sim A_\nu := -i \sqrt{\frac{2}{\pi \nu}} \br*{\frac{e k X}{2 \nu}}^{-\nu} \qquad \text{as} \qquad \nu \to \infty.
\end{equation}
We therefore seek solutions of \eqref{eq:BVP-main} which take the following form on $\Omout$,
\begin{equation}\label{eq:u-target-form}
 u_q(x;k) :=  \sum_{\alpha \in \Z} q_\alpha \frac{1}{A_{|\alpha|}} H^{(1)}_{|\alpha|}(k |x|) e_{\alpha}(\theta(x)), \qquad  q \in \ell^1(\Z),
\end{equation}
where the summability condition on $q$ ensures convergence of the sum.
The solutions $u_q(x;k)$, $x \in \Omout$, are defined  and form analytic functions  for $k \in \C \bs (-\infty, 0]$.
Furthermore, $u_q(\cdot;k)$ are square-integrable if and only if $k \in \C_+$ (in fact, $u_q(\cdot;k)$ is exponentially decaying as $|x|\to \infty$ if $k \in \C_+$ and exponentially growing if $k \in \C_-$).

\begin{de}\label{def:resonance-intuitive}
 We call a number $k \in \C \bs (-\infty,0]$ a \emph{scattering resonance} of $H$ if there exists a non-zero solution of boundary value problem \eqref{eq:BVP-main} taking the form $u_q(\cdot;k)$ in $\Omout$.
\end{de}

\begin{remark}
In particular, the (square roots of) eigenvalues of $H$ correspond to scattering resonances in the upper half plane as per definition \ref{def:resonance-intuitive} (note that $H$ happens not to have any eigenvalues in our setting).
\end{remark}

\begin{remark}
	Definition \ref{def:resonance-intuitive} for scattering resonances can be seen to be equivalent to the definition via meromorphic  continuation of the resolvent (Definition \ref{def:resonance}). Indeed, scattering resonances, as per the definition via meromorphic continuation, are characterised as the points for which there exists a corresponding \emph{resonant state} \cite[Theorem 4.7, Definition 4.8]{resonances}. In turn, resonant states can be shown \cite[Definition 4.9]{resonances} to be characterised as the functions $u \in L^2_{\mathrm{loc}}(U^c)$ with $\chi u \in D(H)$ such that 
	\begin{equation}
		(-\Delta - k^2) u = 0 \qquad \text{on} \qquad U^c, 
	\end{equation}
	and there exists compactly supported $g \in L^2(\R^2)$ such that
	\begin{equation}
		u|_{B_X(0)^c} = R_0(k) g |_{B_X(0)^c},
	\end{equation}
	where $R_0(k)= (- \Delta - k^2)^{-1}$ denotes the analytic continuation of the free resolvent on $L^2(\R^2)$. 
	It can then be readily seen that $u$ is a resonant state corresponding to $k \in \C_- \bs (-\infty, 0]$ if and only if it solves \eqref{eq:BVP-main} and takes the form $u_q(\cdot;k)$ in $\Omout$.

\end{remark}

\subsection{Operator theoretic characterisation}
Next, we express scattering resonances in the lower half plane in terms of a natural interface condition that will form the basis of our numerical method below. Introduce the diagonal operator on $L^2(\Gamma)$,
\begin{align}\label{eq:N_def}
	\cN := \diag(\max\{|\alpha|,1\},\alpha\in\Z).
\end{align}
We define fractional Sobolev spaces on $\Gamma$ as
\begin{equation}
 H^s(\Gamma):= \cN^{-s} L^2(\Gamma), \qquad s \in \R.
\end{equation}

The \emph{inner  NtD operator} $\Min(k)$ is defined as
\begin{equation*}
  \Min(k) g := \gamma_\Gamma u_{\mathrm{in}} \qquad k \in \C_+, \quad g \in H^{-\f12}(\Gamma),
\end{equation*}
where the function $u_{\mathrm{in}} \in \HND :=  \set{u|_{\Om} : u \in H^1_0(U^c)}$ is defined
as the weak solution to the boundary value problem
\begin{equation}
  \label{eq:BVP-in}
  \begin{cases}
   - \Delta u_{\mathrm{in}} = k^2 u_{\mathrm{in}} \quad \text{on} \quad \Om \\
    u_{\mathrm{in}}|_{\partial U} = 0,\quad \partial_\nu|_{\Gamma} u_{\mathrm{in}} = g
  \end{cases}.
\end{equation}
Here, $\gamma_\Gamma$ denotes the trace operator for $\Gamma$.
As is well known, $\Min$ forms an analytic family of bounded operators from $H^{-\f12}(\Gamma)$ to $H^{\f12}(\Gamma)$ on $\C_+$ and admits meromorphic continuation to $\C$ which, in particular, is analytic in $\C_-$.

Furthermore, for $k \in \C \bs (-\infty, 0]$, introduce the diagonal operators
\begin{equation}
 N_1(k) :=  \diag \br*{\frac{1}{A_{|\alpha|}} H^{(1)}_{|\alpha|}(kX),\alpha\in\Z}\quad\text{and}\quad  N_2(k) :=  \diag \br*{\frac{1}{A_{|\alpha|}} (H_{|\alpha|}^{(1)})'(kX),\alpha\in\Z}.
\end{equation}
$N_1(k)$ is a bounded operator from $H^{\f12}(\Gamma)$ to $H^{\f12}(\Gamma)$ whereas, by the large order asymptotics for derivatives of Hankel functions of the first kind,
\begin{equation}
 (H_{\nu}^{(1)})'(kX) \sim - \frac{\nu}{kX} A_\nu \qquad \text{as} \qquad \nu \to \infty,
\end{equation}
$N_2(k)$ is a bounded operator from $H^{\f12}(\Gamma)$ to $H^{-\f12}(\Gamma)$. By definition, we have
\begin{equation}\label{eq:N1-N2-prop}
 N_1(k) q = \gamma_\Gamma u_q(\cdot; k) \qquad \text{and} \qquad N_2(k) q = \partial_\nu|_{\Gamma} u_q(\cdot; k)
\end{equation}

A function $u$ solves the BVP \eqref{eq:BVP-main} if and only if
\begin{equation}
  \label{eq:BVP-out-Min}
  \begin{cases}
   - \Delta u = k^2 u \quad \text{on} \quad \Omout \\
    u|_{\Gamma} = - \Min(k)g ,\quad \partial_\nu|_{\Gamma} u = g
  \end{cases}
\end{equation}
for some $g \in H^{-1/2}(\Gamma)$ hence, in particular, satisfies $u|_{\Gamma} = - \Min(k)\partial_\nu|_{\Gamma} u$.
Therefore, a number $ k \in \C_- $ is a scattering resonance of $H$ (as per Definition \ref{def:resonance-intuitive}) if and only if
\begin{equation}\label{eq:H12-op-resonance-cond}
 \exists\, q \in H^{\f12}(\Gamma): \qquad (N_1(k)+ \Min(k) N_2(k))q = 0.
\end{equation}
We may convert the operator in \eqref{eq:H12-op-resonance-cond} to an equivalent operator on $L^2(\Gamma)$ by bordering with appropriate powers of $\cN$,
\begin{equation}\label{eq:T-defn}
T(k) := \frac{1}{2}\cN^{\f12} (N_1(k)+ \Min(k) N_2(k)) \cN^{-\f12}, \qquad k \in \C_-.
\end{equation}
Note that in Appendix \ref{app:decomp}, we prove that $T(k)$ is Fredholm of index zero.
 We arrive at our operator theoretic characterisation for scattering resonances of $H$.
 
\begin{lemma}\label{def:resonance-operator}
 A number $k \in \C_-$ is a scattering resonance of $H$ if $\ker \, T(k) \neq \set{0}$.
\end{lemma}

\subsection{Algorithm for computing scattering resonances}\label{subsec:num-method}

We shall numerically compute scattering resonances by approximating the operator $T(k)$ by a matrix $T_n(k)$ with computable matrix elements.
The index $n$ controls the accuracy of the computation and we anticipate convergence as $n \to \infty$.
There are several simultaneous processes that take place, each giving rise to sources of error.
\begin{itemize}
 \item Finite truncation of the infinite matrix $T(k)$ (expressed in the basis $\set{e_\alpha}$).
 \item Polygonal approximation of the obstacle $U$.
 \item Finite element approximation (FEM) of the inner NtD map $\Min(k)$.
 \item The FEM approximation of $\Min(k)$ is itself efficiently approximated by a truncated eigenfunction expansion and a trick known as \emph{Aitken's acceleration}.
\end{itemize}
Linking these different approximations in a way that ensures convergence is non-trivial and is clarified by the assumptions of our main result Theorem \ref{th:main} below as well as numerical experiments in Section \ref{sec:numerics}.

\subsubsection*{Matrix elements and truncation}
Fix a sequence of parameters $N_n \in \N$, $n \in \N$ which control the  number basis elements taken in matrix truncation.
Consider the orthogonal projection
\begin{equation}\label{eq:Pn}
P_n : L^2(\Gamma) \to \mathrm{span}_\C\set{e_{-N_n},...,e_{N_n}}.
\end{equation}
The inner NtD operator may be expressed as infinite matrices with matrix elements
\begin{equation}
 a_{\alpha\beta}(k) := \br{\Min(k)e_\alpha,e_\beta }_{L^2(\Gamma)},\qquad k \in \C_-, \quad \alpha, \beta \in \Z.
\end{equation}
Then, $P_n T(k) P_n$ is a $(2N_n + 1) \times (2N_n + 1) $ matrix with matrix elements
\begin{equation}\label{eq:T-ab}
T_{\alpha \beta} (k) = \frac{\delta_{\alpha \beta}}{2 A_{|\alpha|}} H^{(1)}_{|\alpha|}(kX) + a_{\alpha \beta}(k) \frac{\max\set{|\alpha|,1}^{-\f12} \max\set{|\beta|,1}^{\f12}}{2A_{|\alpha|}} \br*{ H^{(1)}_{|\alpha|-1}(kX) - \frac{|\alpha|}{kX}H^{(1)}_{|\alpha|}(kX)}
\end{equation}
where $\delta_{\alpha \beta}$ denotes the Kronecker delta symbol and we used the formula
\begin{equation}
 (H^{(1)}_{\nu})'(z) = H^{(1)}_{\nu - 1}(z) - \frac{\nu}{z} H^{(1)}_{\nu}(z).
\end{equation}
The Hankel functions in $T_{\alpha \beta}(k)$ may be efficiently numerically computed using known methods, however, the computation of $a_{\alpha \beta}(k)$ requires more work.

\subsubsection*{Geometric approximation}
Let $U_n\subset B_{X-1}(0)$, $n \in \N,$ denote a sequence of closed, bounded sets with polygonal boundaries approximating the obstacle $U$. Let $B^n_X \subset B_X(0)$ be a sequence of convex, polygonal domains approximating $B_X(0)$  in the sense that
\begin{equation}
  B_X(0) = \bigcup_{n=1}^\infty B_X^n,
\end{equation}
such that the corners of $B^n_X$ lie on $\Gamma$.
Polygonal approximations for the inner domain and interface may then be defined as
\begin{equation}
  \Om_n := B_X^n \bs U_n \qquad \text{and} \qquad \Gamma_n := \partial B_X^n, \qquad n \in \N.
\end{equation}

\subsubsection*{FEM approximation at a fixed point}
 For each $n \in \N$, let $\cT_n$ be a triangulation of $\Om_n$ and consider the P1 finite element spaces
 \begin{align}
  \label{eq:Vh-defn}
  V^n(\Om_n) & := \set*{u \in C(\Om_n): u|_T \text{ affine for all }T \in \cT_n}, \\
\label{eq:VhND-defn}
  \VhND & := \set*{u \in V^n(\Om_n): u|_{\partial U_n} = 0},
\end{align}
where $C(\Om_n)$ denotes the space of continuous, complex-valued functions on $\Om_n$. The inner NtD may be approximated at a fixed point $k_0 \in \C_-$ by an operator $\hMinn(k_0)$ defined by
\begin{equation}
 \hMinn(k_0)g = \gamma_{\Gamma_n} u_{\mathrm{in}}^n, \qquad g \in P_n H^{-\f12} (\Gamma)
\end{equation}
where $u^n_{\mathrm{in}}$ is the FEM approximation of the BVP \eqref{eq:BVP-in} for $u_{\mathrm{in}}$ on the mesh $\cT_n$, that is,
\begin{equation}
\label{eq:u-in-n-defn}
u_{\mathrm{in}}^n \in \VhND:\qquad \inner{\nabla u^n_{\mathrm{in}}, \nabla \phi}_{L^2(\Om_n)} = k_0^2 \inner{ u^n_{\mathrm{in}}, \phi}_{L^2(\Om_n)} + \br{\hat \Pi_n g,\phi|_{\Gamma_n}}_{L^2(\Gamma_n)}, \qquad\forall \phi \in \VhND,
\end{equation}
where $\hat \Pi_n$ is a linear interpolation operator taking continuous functions on $\Gamma$ to piecewise affine functions on $\Gamma_n$ (see \eqref{eq:hat-Pi} for a precise definition).
Let $\hat{a}_{\alpha \beta}^{n}(k)$ be the approximation for $a_{\alpha \beta}(k)$ defined by
\begin{equation}
 \hat{a}_{\alpha \beta}^{n}(k) := (\hMinn(k)e_\alpha, \hat{\Pi}_n  e_\beta)_{L^2(\Gamma_n)}.
\end{equation}

\subsubsection*{Eigenfunction expansion}Performing a FEM computation for every spectral parameter $k$ that we wish to test would be extremely expensive, hence we express $\hat{a}_{\alpha \beta}^{n}(k)$ in terms of an eigenfunction expansion.
This gives an expression which is explicit in $k$ hence only a single FEM computation need to be performed. Since there are a very large amount terms in the eigenfunction expansion, we make a further approximation by truncating the sum.

Let $- \DND$ denote the Laplacian on $L^2(\Om)$ endowed with homogeneous Dirichlet boundary conditions on $\del U$
and homogeneous Neumann boundary conditions on $\Gamma$.
Consider the FEM approximations for the eigenvalues and eigenfunctions (respectively) of $-\DND$ on the mesh $\cT_n$, that is, solutions of
\begin{equation}\label{eq:fem-eigs}
\mu_m^n \in \R \textrm{ and } w_m^n \in \VhND: \qquad
\begin{cases}
   \inner{\nabla w_m^n, \nabla \phi}_{L^2(\Om_n)} = \mu_m^n  \inner{w_m^n,\phi}_{L^2(\Om_n)}, \qquad \forall \phi \in \VhND \\
   \norm{w_m^n}_{L^2(\Om_n)} = 1
\end{cases},
\end{equation}
where $m$ runs from $1$ to $d_n := \dim \VhND$.

Fix a sequence of parameters $J_n \in \set{1,...,d_n}$, $n \in \N$, controlling the number of elements in the eigenfunction expansion sum approximation.
As shown in Lemma \ref{lem:expansion-for-S_n}, the approximations $ \hat{a}_{\alpha \beta}^{n}(k)$ may be expressed in terms of an eigenfunction expansion,
\begin{align}
 \hat{a}_{\alpha \beta}^{n}(k) & = \sum_{m=1}^{d_n} \frac{1}{\mu_m^n - k^2} (\hat \Pi_n  g, \gamma_{\Gamma_n}w_m^n)_{L^2(\Gamma_n)} (\gamma_{\Gamma_n} w^n_m, \hat\Pi_n e_\beta)_{L^2(\Gamma_n)} \nonumber \\
  & = \hat{a}_{\alpha \beta}^{n}(k_0)  + \sum_{m=1}^{d_n} \frac{k^2 - k_0^2}{(\mu^n_m - k^2)(\mu^n_m - k_0^2)}(\hat \Pi_n e_\alpha,\gamma_{\Gamma_n} w^n_m)_{L^2(\Gamma_n)}(\gamma_{\Gamma_n} w^n_m, \hat\Pi_n e_\beta)_{L^2(\Gamma_n)}
\end{align}
where the second line follows by a simple computation.
The sum in the second line converges faster than the first hence is more amenable to approximation.
Let $a_{\alpha \beta}^{n}(k)$ be the approximation for $a_{\alpha \beta}(k)$ obtained by truncating the sum in the above expression
\begin{equation}
\label{eq:a-n-matrix-elements}
a^n_{\alpha\beta}(k) := \hat{a}_{\alpha \beta}^{n}(k_0)  + \sum_{m=1}^{J_n} \frac{k^2 - k_0^2}{(\mu^n_m - k^2)(\mu^n_m - k_0^2)}(\hat \Pi_n e_\alpha,\gamma_{\Gamma_n} w^n_m)_{L^2(\Gamma_n)}(\gamma_{\Gamma_n} w^n_m, \hat\Pi_n e_\beta)_{L^2(\Gamma_n)}.
\end{equation}
We denote by $P_n \Min^n(k) P_n$ the $(2N_n + 1) \times (2N_n + 1)$ matrix with matrix elements $a^n_{\alpha \beta}(k)$, i.e.,
\begin{equation}
\label{eq:Minn-defn}
 \br{\Min^n(k)e_\alpha,e_\beta}_{L^2(\Gamma)} = a^n_{\alpha\beta}(k), \qquad k \in \C_-, \quad \alpha, \beta \in \set{-N_n,...,N_n}.
\end{equation}

\subsubsection*{Numerical approximation of resonances}

We approximate the operator $T(k)$ by the matrices
\begin{equation}\label{eq:T-n-k-defn}
T_n(k) :=  \frac{1}{2} P_n \cN^{\f12} (N_1(k)+ \Min^n(k) N_2(k)) \cN^{-\f12}  P_n, \qquad  n \in \N, \quad k \in \C_-.
\end{equation}
In other words, we truncate $T(k)$ to a finite matrix and replace  $a_{\alpha \beta}(k)$ by $a^n_{\alpha \beta}(k)$ in expression \eqref{eq:T-ab} for the matrix elements. The points $k$ where $\ker T_n(k) \neq \emptyset$ coincide with the zeros of the function
\begin{equation}
g_n(k) := \det T_n(k), \qquad n \in \N,  \quad k \in \C_-.
\end{equation}
The function $g_n$ are explicitly defined and analytic on $\C_-$.

\begin{center}
\fbox{\parbox{14cm}{\centering The zeros of the function $g_n(k)$, serve as numerical approximations for the scattering resonances of $H$.
}}
\end{center}
\noindent

\begin{remark}\label{rem:lev-marl}
	The algorithm presented in \cite{LevitinMarletta} consists in essentially the same approximation procedure applied to the operator $\Min(k) + \Mout(k)$ instead of $N_1(k) + \Min(k) N_2(k)$, where $\Mout(k)$ is the meromorphic continuation of the NtD operator for the Laplacian on $\R^2 \bs B_X(0)$. The advantage of our approach is that we approximate scattering resonances as zeros of an analytic function, whereas the corresponding function \cite{LevitinMarletta} has poles in general, arising due to the poles of $\Mout(k)$. 
\end{remark}

\subsection{Main convergence result}\label{subsec:main-conv}

Our sole geometric assumption on the obstacle  $U \subset \R^2$ is the following.
\begin{assumption}\label{ass:U}
$U$ belongs to the set $\cS$ described in Section \ref{subsec:comp}
\end{assumption}
Furthermore, the approximations $U_n$ and $B_X^n$, $n \in \N$, for the obstacle and the ball must converge geometrically in the following sense.
Examples of admissible approximations are given in Section \ref{subsec:examples}.  

\begin{assumption}\label{ass:Un}
 We have
\begin{equation}\label{eq:Un-conv}
  \dist_H(\partial U_n, \partial U) + \dist_H(B_X(0) \bs U_n, B_X(0) \bs U) \to 0 \quad \text{as} \quad n \to \infty,
\end{equation}
where $\dist_H$ denotes the Hausdorff distance.

\end{assumption}

The following standard assumption for the triangulation $\cT_n$ shall be supposed. Let $|e|$ denote the length of an edge $e$ and $|T|$ denote the area of an element $T$.
  \begin{assumption}[Shape regularity]\label{ass:shape-reg}
    There exists a constant $C_\theta > 0$ independent of $n$ such that for any element $T \in \cT_n$
    and any edge $e \subset T$
    \begin{equation}
     \frac{1}{C_\theta}\leq \frac{|T|^{1/2}}{|e|} \leq C_\theta.
    \end{equation}
  \end{assumption}

  Denote
  \begin{equation}
    h_n := \max_{T \in \cT_n} \diam(T), \qquad n \in \N.
  \end{equation}
  In order to ensure convergence, we need to balance the parameters $N_n$, $h_n$ and $J_n$ in the limit $n \to \infty$.
  \begin{assumption}\label{ass:param}
    The following limits hold
    \begin{equation*}
     \lim_{n \to \infty} N_n = \infty, \qquad  \lim_{n \to \infty} h_n N^{\f54}_n = 0  \qquad \text{and} \qquad \lim_{n \to \infty} J_n N_n^{-2} = \infty.
    \end{equation*}
  \end{assumption}

  \begin{theorem}\label{th:main}
    Suppose that Assumptions \ref{ass:U}-\ref{ass:param} hold.
    \begin{enumerate}[label=(\alph*)]
    \item
     For any resonance $k \in \C_-$ of $H$, there exists $n_0 \in \N$ and a sequence of zeros $k_n\in \C_-$ of $g_n$, $n \geq n_0$, such that $k_n \to k$ as $n \to \infty.$
    \item
    For any bounded, open subset $D \subset \C_-$ such that $\overline{D}$ does not contain any resonances of $H$, there exists $n_0 \in \N$ such that
    $D$ does not contain any zeros of $g_n$ for any $n \geq n_0$.
\end{enumerate}
  \end{theorem}

In Section \ref{sec:SCI-proof}, we use this result in conjunction with an algorithm for computing zeros of analytic functions with a-priori error control (developed in Section \ref{subsec:compute-zeros}) to prove Theorem \ref{th:SCI}. 
We reformulate Theorem \ref{th:main} in a slightly more general and precise way in Theorem \ref{th:main}', after introducing Mosco convergence and the notion of multiplicity for scattering resonances.

\begin{figure}[htbp]
	\centering
	\includegraphics[width=0.3\textwidth]{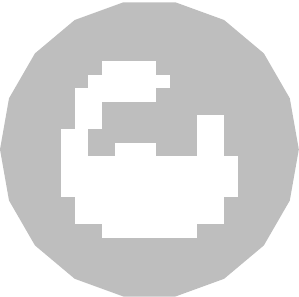}
	\hspace{2cm}
	\includegraphics[width=0.3\textwidth]{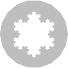}
	\caption{Sketch of a pixelation approximation (left) and pre-fractal approximation (right) of $\Omega$.}
	\label{fig:pixelation}
\end{figure}

\subsection{Examples}\label{subsec:examples}
  In \cite{spectralpaper}, we studied the following general approximation scheme, which is able to produce easy-to-triangulate approximations of $U$,
converging in the sense of Assumption \ref{ass:Un}, provided we have access to the information of whether a given point $x$ lies inside $U$.

\begin{example}[Pixelated domains]\label{ex:pixelation}
Consider an obstacle $U$ satisfying Assumption \ref{ass:U},
Define the \emph{pixelation approximation} of $U$ by (see Figure \ref{fig:pixelation})
\begin{equation}\label{eq:pixels-defn}
 U_n := \bigcup_{j \in L_n} (j + [-\tfrac{1}{2n},\tfrac{1}{2n}]^2), \qquad n \in \N,
\end{equation}
where
\begin{equation}\label{eq:Ln}
L_n := \set{j \in (\tfrac{1}{n}\Z)^2 : j \in U}.
\end{equation}
Then, by \cite[Proposition 4.14]{spectralpaper},  Assumption \ref{ass:Un} for $U_n$ holds true.
\end{example}

In addition, many fractals are naturally defined as the Hausdorrf limit of a sequence of polygonal ``pre-fractal'' sets.
In the case where $\partial U$ is such a fractal, one may use these pre-fractals to construct $U_n$. As a simple example, consider the Koch snowflake.
\begin{example}[Koch snowflake]
Consider the case that the obstacle $U$ is the Koch snowflake, which is defined as the Hausdorrf limit of a sequence of pre-fractals $U_n$, as illustrated in Figure \ref{fig:pixelation}. The boundary of the Koch snowflake is path-connected since it naturally inherits a parametrisation from the boundaries of the pre-fractals, hence Assumption \ref{ass:U} is satisfied.
Furthermore, Assumption \ref{ass:Un} for $U_n$ is satisfied, essentially, by construction.
\end{example}

\subsection{Structure of paper} Our paper is organised as follows:
\vspace{3pt}
\noindent
\emph{Section \ref{sec:prelim}:}  Mosco convergence and Gohberg-Sigal theory are introduced, then slightly more general version of our main result is stated.

\vspace{3pt}
\noindent
\emph{Section \ref{sec:main-conv}:} Theorem \ref{th:main} is proved.

\vspace{3pt}
\noindent
\emph{Section \ref{sec:SCI-proof}:} Theorem \ref{th:SCI} is proved.

\vspace{3pt}
\noindent
\emph{Section \ref{sec:numerics}:} An implementation of the numerical method is described with figures illustrating the scattering resonances of fractal obstacles. 

\vspace{3pt}
\noindent
\emph{Appendix \ref{app:decomp}:} Some necessary properties of the operator $T(k)$ are proved.

\vspace{3pt}
\noindent
Throughout, $C > 0$ shall denote a positive constant that may change from line to line and whose dependence shall be indicated throughout unless specified otherwise (e.g. $C_{s,X}(k)$ depends only on $s$, $X$ and $k$).

\section{Preliminaries and reformulation of the result}\label{sec:prelim}

In this section, we collect the necessary tools we require from Mosco convergence and Gohberg-Sigal theory.  Following this we shall reformulate the main result slightly and reduce its proof of Proposition \ref{prop:Tn}, (which essentially states the locally uniform operator norm convergence of $T_n(k)$ to $T(k)$.
\subsection{Mosco convergence}
The notion of Mosco convergence for Hilbert space plays a central role in our analysis.

\begin{de}\label{def:mosco}
  Let $\cW$ and $\cW_n$, $n \in \N$, be closed subspaces of a Hilbert space $\cH$.
  We say that $\cW_n$ converges to $\cW$ in the Mosco sense as $n \to \infty$, denoted by $\cW_n \Mto \cW$ as $n \to \infty$, if the following holds:
  \begin{enumerate}[label=(\roman*)]
  \item
   For every  $u \in \cW$, there exists $u_n \in \cW_n$, $n \in \N$, such that $\norm{u_n - u}_{\cH} \to 0$ as $n \to \infty$.
\item
   For every subsequence $(\cW_{n_j})_{j \in \N}$ of $(\cW_n)_{n \in \N}$, and every sequence $u_{n_j} \in \cW_{n_j}$ with $u_{n_j} \wto u$
  as $j \to \infty$ for some $u \in \cH$, we have $u \in \cW$.
\end{enumerate}
\end{de}

In \cite{spectralpaper}, we studied the case
\begin{equation}
  \label{eq:spectral-paper}
  \cH = H^1(\R^2), \quad \cW = H^1_0(\cO), \quad \cW_n = H^1_0(\cO_n),
\end{equation}
where $\cO \subset \R^2$ and $\cO \subset \R^2$, $n \in \N$, are bounded domains.

\begin{theorem}[{\cite[Theorem 2.3]{spectralpaper}}]
\label{th:spectral-paper}
If
\begin{enumerate}[label=(\roman*)]
\item $\cO$ is topologically regular (i.e., $\cO = \intt(\overline{\cO})$),
\item $\partial \cO$ has zero Lebesgue measure and a finite number of separated, path-connected components,
\item $\partial \cO_n$ is locally connected for all $n \in \N$ with
  \begin{equation}
    \label{eq:dist_H-cO}
    \dist_H(\cO,\cO_n) + \dist_H(\partial \cO, \partial \cO_n) \to 0 \quad \text{as} \quad n \to \infty,
  \end{equation}
\end{enumerate}
then we have
\begin{equation*}
  H^1_0(\cO_n) \Mto H^1_0(\cO) \quad \text{as} \quad n \to \infty.
\end{equation*}

\end{theorem}

An application of this result to the setting in the present paper  gives the following result. This is straightforward to verify, however, for the convenience of the reader we provide a proof.
\begin{lemma}
\label{lem:geom-to-mosco}
   Assumptions \ref{ass:U} and \ref{ass:Un} imply that
  \begin{equation*}
    H^1_0(U_n^c) \Mto H^1_0(U^c) \quad \text{as} \quad n \to \infty.
  \end{equation*}
\end{lemma}
\begin{proof}
  By Theorem \ref{th:spectral-paper}, under Assumptions \ref{ass:U} and \ref{ass:Un}, we have
  \begin{equation}\label{eq:H10-BX}
    H^1_0(B_X \bs U_n) \Mto H^1_0(B_X \bs U) \quad \text{as} \quad n \to \infty.
  \end{equation}
  Let $(\chi_1$, $\chi_2)$ be a partition of unity for the open cover $(B_X$, $\R^2 \bs\overline{B}_{X-1})$ of $\R^2$.
  We then necessarily have $\chi_1 \equiv 1$ on $\overline{B}_{X-1}$.
  Focusing on Mosco convergence condition (i), let $u \in H^1_0(U^c)$.
  Then there exists $\tilde{u}_n \in H^1_0(B_x \bs U_n^c)$ such that $\tilde{u}_n \to \chi_1 u$ in $H^1$,
  hence $u_n := \tilde{u}_n + \chi_2 u \to \chi_1 u + \chi_2 u = u$ in $H^1$.
  Focusing on  condition (ii), let $u_{n_j} \in H^1_0(U_{n_j}^c)$, $j \in \N$, such that $u_{n_j} \wto u$ in $H^1$ for some $u \in H^1(\R^2)$.
  Then $\chi_1u_{n_j} \wto \chi_1 u$ in $H^1$ so $\chi_1 u \in H^1_0(B_X \bs U)$. Consequently, $u = \chi_1 u + \chi_2 u \in H^1_0(U^c)$.
\end{proof}

The following lemma shall be used in Lemma \ref{lem:fem-mosco} to establish Mosco convergence properties for $H^1_{\textrm{ND}}$ spaces and the associated finite element spaces.
\begin{lemma}[{\cite[Lemma 2.4]{CW21}}]\label{lem:W_n+}
  Let $\cW$ and $\cW_n$, $n \in \N$, be closed subspaces of a Hilbert space $\cH$. Suppose that the following holds:
  \begin{enumerate}[label=(\roman*)]
  \item There exists a dense subspace $\tilde{\cW} \subset \cW$ such that for every $u \in \tilde{\cW}$, there exists a sequence $u_n \in \cW_n$,
    $n \in \N$, with $\norm{u_n - u}_\cH \to 0$.
  \item There exists a sequence of closed subspaces $\cW_n^+$ of $\cH$ such that $\cW_n \subset \cW_n^+$ for all $n$ and $\cW_n^+ \Mto \cW$.
  \end{enumerate}
  Then, $\cW_n \Mto \cW$ as $n \to \infty$.
\end{lemma}

\subsection{Gohberg-Sigal theory}\label{subsec:goh-sig}
Let $\cX$ be a Banach space and let $\Sigma \subset \C$ be open.
Let $K : \Sigma \to \cB(\cX)$ be an analytic operator-valued function such that $K(z)$ is compact for every $z \in \Sigma$.
Then $A(z) := I + K(z)$ is Fredholm of index zero.
Furthermore, assume there exists $\tilde{z} \in \Sigma$ such that $A(\tilde{z})$ is invertible.
Then by the analytic Fredholm theorem, $z \mapsto A(z)^{-1}$ is a meromorphic operator-valued function on $\Sigma$ with poles of finite rank.
These pole exactly coincide with the zeros of $A$, in the sense of the following  definition.
\begin{de}\label{def:zeros}
  Let $\cX$ be a Banach space and $\Sigma \subset \C$ be open.
  The \emph{zeros} of an analytic operator-valued function $A:\Sigma \to \cB(\cX)$, are defined as the points $z \in \Sigma$ such $\ker A(z) \neq \set{0}$.
\end{de}
The following factorisation theorem is at the heart of Gohberg-Sigal theory and allows us to define a notion of multiplicity to zeros of operator-valued functions. Note that we only state a simplified version here.

\begin{theorem}[{\cite[Th. 1.8]{ammari2009layer}}]\label{th:factorisation}
  Let $A$ be as above. Then for any $z_0 \in \Sigma$, there exists:
  \begin{enumerate}[label=(\alph*)]
  \item analytic operator-valued functions $W_1,W_2: \Sigma \to \cB(\cX)$, invertible near $z_0$,
  \item mutually disjoint projection operators $P_0,...,P_M$ with
    \begin{equation*}
      \mathrm{rank} P_j = 1 \quad \text{for} \quad 1 \leq j \leq M  \quad \text{and} \quad P_0 + \sum_{j=1}^M P_j = I
    \end{equation*}
  \item positive integers $l_1,...,l_M$,
  \end{enumerate}
  such that
  \begin{equation*}
    A(z) = W_1(z) \br*{P_0 + \sum_{j=1}^M (z-z_0)^{l_j}P_j} W_2(z).
  \end{equation*}
\end{theorem}

\begin{remark} Note that the above result follows from \cite{ammari2009layer} because, in the language used there, $z_0$ is a normal point of $A$
  (cf. \cite[Section 1.1.4]{ammari2009layer})

\end{remark}

\begin{de}\label{def:mult}
  Suppose that $A$ satisfies the hypotheses of Theorem \ref{th:factorisation}.
  The \emph{null multiplicity} of a zero $z_0 \in \Sigma$ of $A$ is defined as $l_1 + \cdots +l_M$.
\end{de}

A simplified version of the generalised Rouch\'e's theorem for operator-valued functions reads as follows.
\begin{theorem}[Generalised Rouch\'e's theorem {\cite[Th. 1.15]{ammari2009layer}}]\label{th:gen:Rouche}
  Let $A_j = I + K_j: \Sigma \to \cB(\cX)$, $j = 1 , 2$, both satisfy the hypotheses of Theorem \ref{th:factorisation}.
  Let $D \Subset \Sigma$ be a simply connected open set with $C^1$ boundary $\partial D$ on which neither $A_1$ nor $A_2$ has any zeros.
  If
  \begin{equation*}
    \norm{A_1(z)^{-1}\br*{A_1(z) - A_2(z)}} < 1 \quad \text{for all} \quad  z \in \partial D,
  \end{equation*}
  then the number of zeros of $A_1$ and $A_2$ in $D$ coincide, counting null multiplicities.
\end{theorem}

\begin{remark}
  Note that the above result follows from \cite{ammari2009layer} because $A_1 - A_2$ is analytic on an open neighbourhood of $\overline{D}$ and,
  in the language used there, $A_1$ is normal with respect to $\partial D$ (cf. \cite[Section 1.3.1]{ammari2009layer}).
\end{remark}

The following generalisation of Hurwitz's theorem follows from the above Rouch\'e's theorem and shall be applied in the proof of our main result.

\begin{lemma}\label{lem:hurwitz}
Let  $A = I + K: \Sigma \to \cB(\cX)$, and  $A_n = I + K_n: \Sigma \to \cB(
)$, $n \in \N$, all satisfy the hypotheses of Theorem \ref{th:factorisation}. If we have
\begin{equation}\label{eq:hur-loc-uni}
 \norm{A(z) - A_n(z)} \to 0 \quad \text{as} \quad n \to \infty \qquad \text{locally uniformly for} \qquad z \in \Sigma,
\end{equation}
then the following holds.
\begin{enumerate}[label = (\alph*)]
 \item For any zero $w$ of $A$ with null multiplicity $\nu$, there exists $n_0 \in \N$ and $\nu$ sequences of zeros $w^{(1)}_n, ..., w^{(\nu)}_n$, $n \geq n_0$, of $A_n$ such that $w^{(j)}_n \to w$ as $n \to \infty$ for all $j \in \set{1,...,\nu}$
 \item For any bounded, open subset $D \Subset \Sigma $ not containing any zeros of $A$, there exists $n_0 \in \N$ such that $A_n$ has no zeros in $D$ for all $n \geq n_0$.
\end{enumerate}

\end{lemma}

\begin{proof}
First, focus on (a). Let $w$ be a zero of $A$. Since $A$ is analytic, there exists $\epsilon_0 > 0$ such that $\overline{B_{\epsilon_0}(w)}$ intersects neither $\partial \Sigma $ nor any other zero of $A$. Let $\epsilon \in (0, \epsilon_0)$. There exists $C > 0$ such that for all $n \in \N$,
\begin{equation*}
 \sup_{|z-w|=\epsilon} \norm{A(z)^{-1}\br*{A(z) - A_n(z)}} \leq C \sup_{|z-w|=\epsilon} \norm{A(z) - A_n(z)}
\end{equation*}
hence, by the hypothesis \eqref{eq:hur-loc-uni} and the generalised Rouch\'e's theorem, there exists $n_1(\epsilon) \in \N$ such that $A_n$ has $\nu$ zeroes in $B_{\epsilon}(w)$ for all $n \geq n_1(\epsilon)$. The proof of (a) is completed by setting $n_0 = n_1(\epsilon)$.

Next, we prove (b). Let $D_1,...,D_N \subseteq D$ be a collection of simply connected, bounded open sets such that $\cup_{j=1}^N D_j = D$. By the same argument used to prove (a), for each $j \in \set{1,...,N}$, there exists $n_j \in \N$ such that $A_n$ does not have any zeros in $D_j$ for all $n \geq n_j$. The proof of (b) is completed by setting $n_0 = \max \set{n_1,...,n_N}$.
\end{proof}

In addition, we shall need the following lemma, which will be used to verify the uniform convergence hypothesis of the above generalised Hurwitz's theorem.

\begin{lemma}
\label{lem:point-to-uniform}
	Let $D\subset\C$ be a convex, bounded Cauchy domain (cf. \cite[p. 268]{taylor-lay}) and let $A,A_n:D\to \cB(\cX)$, $n\in\N$, be analytic operator-valued functions. Assume that
	\begin{enumerate}
		\item[(i)] $\|A_n(z)-A(z)\| \to 0$ as $n\to\infty$ for all $z\in D$,
		\item[(ii)] $\sup_{n\in\N}\sup_{z\in D} \|A_n(z)\| <+\infty$.
	\end{enumerate}
	then
	\begin{align*}
		\sup_{z\in Q} \|A_n(z)-A(z)\|  \to 0 \quad \text{ as }n\to\infty
	\end{align*}
	for any compact $Q\subset D$.
\end{lemma}

\begin{proof}
For $z,w\in\C$ the notation $[z,w]$ denoted the line segment connecting $z$ to $w$.
	The strategy of the proof is to prove equicontinuity of the sequence $(\|A_n\|)_{n\in\N}$  and then use the Arzel\'a-Ascoli theorem (note that boundedness already follows from the assumptions).

	For convenience we denote $C_n := A_n - A$ and
	\begin{align*}
		f_n(z) := \|C_n(z)\| .
	\end{align*}
	in the rest of the proof. For $z,w\in D$ one has
	\begin{align}\label{eq:lbound_C_2}
		|f_n(z)-f_n(w)| \leq \|C_n(w) - C_n(z)\|
		= \biggl\| \int_{[z,w]} \f{dC_n}{dz}(\tilde w)\,d \tilde w \biggr\|
		\leq |z-w| \biggl\| \f{dC_n}{dz}(\cdot) \biggr\|_{L^\infty([z,w];\cB(\cX))}
	\end{align}
	for all $n\in\N$. Note that convexity of $D$ implies $[z,w]\subset \operatorname{conv}(Q) \Subset D$. Since $w\mapsto \big\|\f{dC_n}{dz}(w)\big\|$ is continuous, it has a maximum on $[z,w]$, say
	\begin{align}\label{eq:lbound_C_3}
		\biggl\| \f{dC_n}{dz}(\cdot) \biggr\|_{L^\infty([z,w];\cB(\cX))} = \biggl\| \f{dC_n}{dz}(\zeta_n) \biggr\| .
	\end{align}
	The openness of $D$ and the fact that $\{\zeta_n\}_{n\in\N}\Subset D$ imply that $r := \inf_{n\in\N} \dist(\zeta_n,\del B)>0$.
	By Cauchy's integral formula for operator-valued functions (cf. \cite[Th. V.1.4]{taylor-lay}) one has
	\begin{align*}
		\f{dC_n}{dz}(\zeta_n) = \f{1}{2\pi i}\oint_{\del D} \f{C_n(w)}{(\zeta_n-w)^2}\,dw.
	\end{align*}
	Taking norms on both sides, we obtain
	\begin{equation}\label{eq:cauchy_bound}
		\biggl\| \f{dC_n}{dz}(\zeta_n) \biggr\|  \leq \f{1}{2\pi} \oint_{\del D} \f{\|C_n(w)\| }{|\zeta_n-w|^2}\,dw
		\leq \f{|\del D|}{r^2} \|C_n\|_{L^\infty(D;\cB(\cX))},
	\end{equation}
	where $|\del D|$ denotes the arc length of $\del D$. By boundedness of $\|A_n\| $ (assumption (ii)) and \eqref{eq:cauchy_bound} we have
	\begin{equation}\label{eq:dCdz<infty}
		\sup_{n\in\N} \biggl\| \f{dC_n}{dz}(\zeta_n) \biggr\|  \leq \f{|\del D|}{r^2} \sup_{n\in\N} \|C_n\|_{L^\infty(\gamma_n;\cB(\cX))}
		< +\infty.
	\end{equation}
	Combining eqs. \eqref{eq:lbound_C_2}-\eqref{eq:dCdz<infty} we conclude that
	\begin{align*}
		|f_n(z)-f_n(w)| \leq C|z-w|
	\end{align*}
	for all $z,w\in Q$ and all $n\in\N$, hence the set $\{f_n|_Q\}_{n\in\N}$ is equicontinuous. Applying the Arzel\'a-Ascoli theorem we conclude that there exists a subsequence $(f_{n_k})_{k\in\N}$, which converges uniformly on $Q$. By pointwise convergence of the $A_n$ (assumption (i)) the limit is identified as 0, i.e. $f_{n_k}\to 0$ uniformly on $Q$. Finally, applying the same reasoning to every subsequence of $(f_{n})_{n\in\N}$ we conclude that $f_{n}\to 0$ uniformly on $Q$.
\end{proof}

Finally, in order to apply Gohberg-Sigal theory to our setting, we require that the operator-valued functions $T$ and $T_n$ admit decompositions as the sum of the identity and a compact operator on $L^2 (\Gamma)$.
In fact, this does not hold for $T_n$ since it is finite-rank, however, we may easily overcome this by introducing an auxiliary operator-valued function
\begin{equation}
\label{eq:tT-n-defn}
 \tilde T_n(k) := (I - P_n) + T_n(k), \qquad n \in \N, \quad k \in \C_-
\end{equation}
whose zeros coincide with $T_n$.
On the other hand, the desired decomposition for $T$ is proven in Appendix \ref{app:decomp}.

\begin{prop}
\label{prop:T-decomp}
Consider the operator-valued function $T: \C_- \to \cB(L^2(\Gamma))$ defined by \eqref{eq:T-defn}.  There exists an analytic compact-operator-valued function $K:\C_- \to \cK(L^2(\Gamma))$ such that
\begin{equation*}
T(k) = I  + K(k), \qquad k \in \C_-.
\end{equation*}
\end{prop}

\subsection{Reformulation of Theorem \ref{th:main}}

We are now in a position to state a slight generalisation of Theorem \ref{th:main}.
This generalisation takes into account the algebraic multiplicities of the resonances,
which are defined as follows.
\begin{de}
The \emph{algebraic multiplicity} of a scattering resonance $k \in \C_-$ of $H$ is defined as the null multiplicity of the corresponding zero of the operator-valued function $T: \C_- \to \cB(L^2(\Gamma))$ defined by \eqref{eq:T-defn}.
\end{de}
 Theorem \ref{th:main} follows immediately from next theorem by Lemma \ref{lem:geom-to-mosco}.

\begin{customthm}{\ref{th:main}'}
    Suppose that Assumptions \ref{ass:Un}-\ref{ass:param} hold and
      \begin{equation*}
    H^1_0(U_n^c) \Mto H^1_0(U^c) \quad \text{as} \quad n \to \infty.
  \end{equation*}
    \begin{enumerate}[label=(\alph*)]
    \item
     For any resonance $k \in \C_-$ of $H$ with algebraic multiplicity $\nu$, there exists $n_0 \in \N$ and $\nu$ sequences of zeros $k^{(1)}_n,...,k^{(\nu)}_n\in \C_-$ of $g_n$, $n \geq n_0$, such that $k^{(j)}_n \to k$ as $n \to \infty$
     for all $j \in \C_-$.
    \item
    For any bounded, open subset $D \subset \C_-$ such that $\overline{D}$ does not contain any resonances of $H$, there exists $n_0 \in \N$ such that
    $D$ does not contain any zeros of $T_n$ for any $n \geq n_0$.
\end{enumerate}
\end{customthm}

The purpose of Section \ref{sec:main-conv} shall be to prove the following result regarding the convergence and uniform boundedness of the sequence of numerical approximations $T_n(k)$, $n \in \N$.
\begin{prop}
\label{prop:Tn}
 Let $T: \C_- \to \cB(L^2(\Gamma))$ and $\tilde T_n: \C_- \to \cB(L^2(\Gamma))$, $n \in \N$, be the operator-valued functions defined by \eqref{eq:T-defn} and \eqref{eq:tT-n-defn} respectively.  Suppose that Assumptions \ref{ass:Un}-\ref{ass:param} hold and
      \begin{equation*}
    H^1_0(U_n^c) \Mto H^1_0(U^c) \quad \text{as} \quad n \to \infty.
    \end{equation*}
    Then the following holds.
 	\begin{enumerate}[label=(\alph*)]
		\item $\|\tilde T_n(k)-T(k)\|_{L^2 \to L^2} \to 0$ as $n\to\infty$ for all $k\in \C_-$,
		\item $\sup_{n\in\N}\sup_{k\in D} \|\tilde T_n(k)\|_{L^2 \to L^2} <+\infty$ for every bounded set $D \subset \C_-$ with $\overline{D} \subset \C_-$.
	\end{enumerate}
\end{prop}

As we shall now show, the main result follows from this proposition.

\begin{proof}[Proof of Theorem \ref{th:main}']
By Proposition \ref{prop:Tn} and Lemma \ref{lem:point-to-uniform}, we have $\norm{\tilde T_n(k)- T(k)}_{L^2 \to L^2} \to 0$ as $n \to \infty$ locally uniformly for $k \in \C_-$. The result then follows by an application of the generalised Hurwitz's result Lemma \ref{lem:hurwitz}.
\end{proof}

\section{Main convergence proof}\label{sec:main-conv}
In this section, we prove Theorem \ref{th:main} which, as shown above, amounts to proving Proposition \ref{prop:Tn}. In Subsection \ref{subsec:initial}, we begin by estimating the full error term $\norm{T(k) - \tilde T_n(k)}_{L^2 \to L^2}$
from Proposition \ref{prop:Tn} (a) in terms of three other terms which, respectively, isolate the error arising from matrix truncation,  sum  truncation in the eigenfunction expansion and the error from the finite element method (including the polygonal approximation of boundaries).
Subsection \ref{subsec:tools} then develops several necessary tools.
The matrix and sum truncation error are dealt with in in Subsection \ref{subsec:mat-sum} while the finite element error is dealt with in Subsection \ref{subsec:mosco-fem}. The proof is concluded in Subsection \ref{subsec:proof}.

\subsection{Initial estimates}
\label{subsec:initial}
Let us first introduce  three operators that will play a key role in our analysis.
\begin{itemize}
 \item
 Define an FEM space on the  boundary $\Gamma_n$ as follows,
\begin{equation*}
V^n(\Gamma_n) := \set*{ u|_{\Gamma_n} : u \in V^n(\Omega_n)}.
\end{equation*}
We let
\begin{equation}
\label{eq:hat-Pi}
\hat \Pi_n : C(\Gamma) \to V^n(\Gamma_n)
\end{equation}
denote a linear interpolation operator on $\Gamma$, which is defined as one would expect: $\hat{\Pi}_n g$ is the unique function in $V^n(\Gamma_n)$ such that $(\hat \Pi_n g)(p) = g(p)$ for any $p \in \Gamma \cap \Gamma_n$
\item
 For each $k \in \C_-$, the operator $$S(k) \in \cB(H^{-\f12}(\Gamma);\HND)$$ is defined as the weak solution operator for the boundary value problem \eqref{eq:BVP-in} associated to $\Min(k)$.
In other words, for any $g \in H^{-\f12}(\Gamma)$, $u = S(k)g$ solves
\begin{equation}\label{eq:weak-S}
 u \in \HND: \qquad \inner{\nabla u , \nabla \phi}_{L^2(\Om )} = k^2\inner{u, \phi}_{L^2(\Omega)} + (g, \phi|_\Gamma)_{H^{-\f12},H^{\f12}}, \qquad \forall \phi \in \HND.
\end{equation}
Note that we have
\begin{equation}
 \Min(k) = \gamma_\Gamma S(k), \qquad k \in \C_-.
\end{equation}
\item
Furthermore, we introduce an FEM discretisation
$$S_n(k) \in \cB(P_nH^{-\f12}(\Gamma);\VhND) $$
of $S(k)$ as in \eqref{eq:u-in-n-defn}, that is, given $g \in P_n H^{-\f12}(\Gamma)$, $u_n  = S_n(k) g$ solves
\begin{equation}
\label{eq:S-n-defn}
u_n \in \VhND:\qquad \inner{\nabla u_n, \nabla \phi}_{L^2(\Om_n)} = k_0^2 \inner{ u_n, \phi}_{L^2(\Om_n)} + \br{\hat \Pi_n g,\phi|_{\Gamma_n}}_{L^2(\Gamma_n)}, \qquad\forall \phi \in \VhND.
\end{equation}
Note that in terms of the approximation $\hMinn(k)$ for the inner NtD map introduced in Section \ref{sec:overview}, we have
\begin{equation*}
 \hMinn(k) = \gamma_{\Gamma_n} S_n(k), \qquad k \in \C_-, \quad n \in \N.
\end{equation*}
\end{itemize}

In what follows, it shall be helpful to consider intermediate sequence of operators
$\hat \Pi^*_n \gamma_{\Gamma_n} \hMinn(k)$, $n \in \N$, on $L^2(\Gamma)$.
Here, $\hat \Pi_n^*$ denotes the adjoint of the interpolation operator $\hat \Pi_n$ (i.e., so that the matrix elements of $\hat \Pi^*_n \hMinn(k)$ are given by $(\hMinn(k) e_\alpha, \hat \Pi_n e_\beta)_{L^2(\Gamma_n)}$).

The three sources of error are as follows.
\begin{itemize}
 \item
The \emph{matrix truncation error} is defined by
 \begin{equation}
  \Emat(k):= \norm{K(k) -  P_n K(k)P_n}_{L^2 \to L^2}
 \end{equation}
 where $K(k)$ is the family of compact operators from Proposition \ref{prop:T-decomp}.
 \item
 The \emph{sum truncation error} is defined by
 \begin{equation}
  \Esum(k) := \norm{P_n(\hat \Pi_n^* \hMinn(k) - \Min^n(k))P_n}_{H^{-\f12} \to H^{\f12}}.
 \end{equation}
 \item
 The \emph{finite element error} is defined by
 \begin{equation}
  \Esol(k) := \norm{P_n(\Min(k) - \hat \Pi_n^* \hMinn(k))P_n }_{H^{-\f12} \to H^{\f12}}.
 \end{equation}
 \end{itemize}
We have the following estimate.
\begin{lemma}
  \label{lem:initial}
  For all $k \in \C_-$ and $n \in \N$, it holds that
  \begin{equation*}
    \norm{T(k) - \tilde T_n(k)}_{L^2 \to L^2} \leq \Emat(k) + C_X(k) \br*{\Esum(k) + \Esol(k)},
  \end{equation*}
  for some locally uniform constant $C_X(k) > 0$.
\end{lemma}

\begin{proof}

First, adding and subtracting $I + P_n K(k) P_n$,
\begin{equation*}
 \norm{T(k) - \tilde T_n(k)}_{L^2 \to L^2} \leq \underbrace{\norm{T(k) - (I + P_n K(k) P_n)}_{L^2 \to L^2 }}_{ = \Emat(k)} + \norm{(I + P_n K(k) P_n) - \tilde T_n(k)}_{L^2 \to L^2 }
\end{equation*}
where we have recognised the first term on the right hand side to be $\Emat(k)$ since
$T(k) = I + K(k)$. The second term may be simplified using the definitions of $\tilde T_n(k)$, $T_n(k)$ and $T(k)$, as well as the locally uniform boundedness of $N_2(k)$ between $H^{\f12}$ and $H^{-\f12}$,
\begin{align*}
 \norm{(I + P_n K(k) P_n) - \tilde T_n(k)}_{L^2 \to L^2 } & = \norm{ P_n T(k) P_n -  T_n(k)}_{L^2 \to L^2 } \\
 & = \tfrac{1}{2} \norm{ \cN^{\f12} P_n (\Min(k) - \Min^n(k))P_n N_2(k)\cN^{-\f12}}_{L^2 \to L^2} \\
 & = C_X(k) \norm{  P_n (\Min(k) - \Min^n(k))P_n}_{H^{-\f12} \to H^{\f12}}.
\end{align*}
The proof is complete by adding and subtracting  $\hat \Pi_n^* \hMinn(k)$ as follows,
\begin{multline*}
 \norm{  P_n (\Min(k) - \Min^n(k))P_n}_{H^{-\f12} \to H^{\f12}} \\\leq \underbrace{\norm{P_n(\hat \Pi_n^* \hMinn(k) - \Min^n(k))P_n}_{H^{-\f12} \to H^{\f12}}}_{= \Esum(k)}
 + \underbrace{\norm{P_n(\Min(k) - \hat \Pi_n^* \hMinn(k))P_n }_{H^{-\f12} \to H^{\f12}}}_{= \Esol(k)}.
\end{multline*}

\end{proof}

\subsection{Bijection, interpolation and extension operators}
\label{subsec:tools}
In this section, we collect some tools that we require, namely the following:
\begin{itemize}
 \item (Lemma \ref{lem:iota-n}) In order to compare functions on the interface $\Gamma$ with the polygonal approximation of the interface $\Gamma_n$, we shall use bijection operators
 $$ \hat{\iota}_n : C(\Gamma_n) \to C(\Gamma), \qquad n \in \N. $$
 \item (Lemma \ref{lem:E-n}) In order to compare piecewise affine functions on $\Om_n$ with functions on $\Om$ we shall use extension operators
 $$E_n: \VhND \to H^1(B_X), \qquad n \in \N. $$
 \item (Lemma \ref{lem:interpolation})  We shall need certain Sobolev norm estimates for the boundary linear interpolation operator $\hat{\Pi}_n$.
 \item (Lemma \ref{lem:trace-uniform}) Finally, we need uniform (in $n$) Sobolev norm estimates for the trace operator $\gamma_{\Gamma_n}$.
\end{itemize}

\red{}

First, we construct the bijection operators $\hat{\iota}_n$.
\begin{lemma}\label{lem:iota-n}
 There exists a sequence of invertible maps $\hat{\iota}_n : C(\Gamma_n) \to C(\Gamma)$, $n \in \N$, such that:
 \begin{enumerate}[label = (\alph*)]
  \item  $(\hat{\iota}_n g)(p) = g(p)$ for any point $p \in \Gamma \cap \Gamma_n$, $n \in \N$ and $g \in C(\Gamma_n)$,
  \item There exists a sequence $\delta_X(n) > 0$, $n \in \N$, with $\lim_{n \to \infty}\delta_X(n)=0$ such that for any $s \geq 0$,  $(g_n^{(1)})_{n \in \N} \subset C(\Gamma_n) \cap H^{-s}(\Gamma_n)$ and
  $(g_n^{(2)})_{n \in \N} \subset C(\Gamma_n) \cap H^{s}(\Gamma_n)$, we have
  \begin{equation*}
   \abs*{(\hat \iota_n g_n^{(1)},\hat \iota_n g_n^{(2)})_{L^2(\Gamma)} - (g_n^{(1)}, g_n^{(2)})_{L^2(\Gamma_n)} } \leq \delta_X(n) \norm{g_n^{(1)}}_{H^{-s}(\Gamma_n)}\norm{g_n^{(2)}}_{H^{s}(\Gamma_n)}, \qquad n \in \N .
  \end{equation*}
  \item For any $s \in [0,1]$, we have
  \begin{equation*}
   \sup_{n \in \N}\, \norm{\hat{\iota}_n}_{H^s \to H^s} < \infty \qquad \text{and} \qquad \sup_{n \in \N}\, \norm{\hat{\iota}^{-1}_n}_{H^s \to H^s} < \infty.
  \end{equation*}

 \end{enumerate}
\end{lemma}
\begin{proof}
 Recall that $\Gamma = \partial B_X(0)$ and , for each $n \in \N$, $\Gamma_n = \partial B_X^n$, where $B_X^n$ is convex with $\set{0} \in B_X^n \subset B_X(0)$. Therefore, there exists a unique Lipschitz bijection $\iota_n: \Gamma \to \Gamma_n$ such that, for any $p \in \Gamma$, $\iota_n(p)$ is the unique point in $\Gamma_n $ with the same angular coordinate as $p$. Define
 \begin{equation}
  \hat \iota_n g = g \circ \iota_n, \qquad g \in C(\Gamma_n).
 \end{equation}
Then, (a) clearly holds.

Focus on (b).
Let $s \geq 0$. Let $(g_n^{(1)})_{n \in \N} \subset C(\Gamma_n)$ be bounded in $H^{-s}(\Gamma_n)$ and let $(g_n^{(2)})_{n \in \N} \subset C(\Gamma_n)$ be bounded in $H^{s}(\Gamma)$.
By a change of variables,
\begin{equation*}
 (\hat \iota_n g_n^{(1)},\hat \iota_n g_n^{(2)})_{L^2(\Gamma)} = \int_{\Gamma} g_n^{(1)}(\iota_n(\theta)) \overline{g_n^{(2)}}(\iota_n(\theta)) \,\d \theta = \int_{\Gamma_n } g_n^{(1)}(t) \overline{g_n^{(2)}}(t) \frac{\d \iota_n^{-1}}{\d t}(t)\, \d t.
\end{equation*}
Consequently, we have
\begin{align*}
 \abs*{(\hat \iota_n g_n^{(1)},\hat \iota_n g_n^{(2)})_{L^2(\Gamma)} - (g_n^{(1)}, g_n^{(2)})_{L^2(\Gamma_n)}}&  =  \abs*{\br*{(1 - \frac{\d \iota_n^{-1}}{\d t} )g_n^{(1)}, g_n^{(2)})}_{L^2(\Gamma_n )}}\\
 & \leq \norm{(1 - \frac{\d \iota_n^{-1}}{\d t} )g_n^{(1)}}_{H^{-s}(\Gamma_n)}
 \norm{g_n^{(2)}}_{H^{s}(\Gamma_n)} \\
 & \leq \norm{1 - \frac{\d \iota_n^{-1}}{\d t} }_{L^\infty(\Gamma_n)}
 \norm{g_n^{(1)}}_{H^{-s}(\Gamma_n)}
 \norm{g_n^{(2)}}_{H^{s}(\Gamma_n)}.
\end{align*}
It can be directly verified that $\tfrac{\d \iota_n^{-1}}{\d t} \to 1$ as $n \to \infty$ pointwise, completing the proof.

Finally, focus on (c). We shall only prove the statement for
$\hat{\iota}^{-1}_n$; the proof of the statement for $\hat{\iota}_n$ is very similar.
The case $s= 0$ follows directly from (b) so it suffices to show that
\begin{equation}
 [\hat{\iota}_n^{-1} u ]_{H^s(\Gamma_n)} \leq C_X [u ]_{H^s(\Gamma)}, \qquad n \in \N, \quad u \in H^s(\Gamma),
\end{equation}
for $s \in (0,1)$, where $[\cdot]_{H^s(\Gamma)}$ and $[\cdot]_{H^s(\Gamma_n)}$ denote the respective $H^s$ semi-norms.
By the Gagliardo representation for the semi-norms $[\cdot]_{H^s(\Gamma)}$, we have
\begin{equation}
 [\hat{\iota}_n^{-1} u]^2_{H^s(\Gamma_n)} \leq C_X \int_{\Gamma_n}\int_{\Gamma_n} \frac{|u(\iota_n^{-1}(x)) - u(\iota_n^{-1}(y))|^2}{|x - y|^{1 + 2s}} \, \d x \d y.
\end{equation}
Performing a change of variables, we have
\begin{align}\label{eq:gagliardo-est}
\int_{\Gamma_n}\int_{\Gamma_n} &\frac{|u(\iota_n^{-1}(x)) - u(\iota_n^{-1}(y))|^2}{|x - y|^{1 + 2s}} \, \d x \d y   = \int_{\Gamma}\int_{\Gamma} \frac{|u(\theta) - u(\phi)|^2}{|\iota_n(\theta) - \iota_n(\phi)|^{1 + 2s}} \frac{\d \iota_n}{\d \theta}(\theta) \frac{\d \iota_n}{\d \phi}(\phi) \, \d \theta \d \phi \nonumber \\
& \leq \norm{\iota_n}^2_{W^{1,\infty}(\Gamma)}\br*{\sup_{\theta,\phi \in \Gamma}\frac{|\theta - \phi|}{|\iota_n(\theta) - \iota_n(\phi)|}}^{1+2s}  \int_{\Gamma}\int_{\Gamma} \frac{|u(\theta) - u(\phi)|^2}{|\theta - \phi|^{1 + 2s}} \, \d \theta \d \phi \nonumber \\
& \leq C_X \br*{\sup_{\theta,\phi \in \Gamma}\frac{|\theta - \phi|}{|\iota_n(\theta) - \iota_n(\phi)|}}^{1+2s} [u ]^2_{H^s(\Gamma)}
\end{align}
where the last line follows by again using the Gagliardo representation for the $H^s$ semi-norm, as well as the fact that $\norm{\iota_n}_{W^{1,\infty}(\Gamma)} \leq C_X$.
Finally, the term in the brackets on the right hand side of \eqref{eq:gagliardo-est} is estimated as
\begin{align}\label{eq:sup-term-est}
 \sup_{\theta, \phi \in \Gamma}\frac{|\theta - \phi|}{|\iota_n(\theta) - \iota_n(\phi)|} = \sup_{x, y \in \Gamma_n}\frac{|\iota_n^{-1}(x) - \iota_n^{-1}(y)|}{|x - y|} \leq C_X \norm{\iota_n^{-1}}_{W^{1,\infty}(\Gamma_n)} \leq C_X,
\end{align}
 where the third line holds by Taylor's theorem. Substituting \eqref{eq:sup-term-est}  into the right hand side of \eqref{eq:gagliardo-est}, we obtain a constant independent of $s$ by take a supremum over $s \in (0,1)$, completing the proof.

\end{proof}

Next, we construct the aforementioned extension operators $E_n$.

\begin{lemma}\label{lem:E-n}
There exists a sequence of operators $E_n: \VhND \to H^1(B_X)$, $n \in \N$, such that:
\begin{enumerate}[label = (\alph*)]
 \item $E_n u = u$ on $\Om_n$ for any $u \in \VhND$ and $n \in \N$,
 \item $\sup_{n \in \N} \norm{E_n}_{H^1 \to H^1} < \infty$,
 \item $\hat{\iota}_n u|_{\Gamma_n} = (E_n u)|_{\Gamma}$ for every $n \in \N$ and $u \in \VhND$.
\end{enumerate}
\end{lemma}

\begin{proof}
Let $\cT_n^{\partial}$ denote the set of elements $T \in \cT_n$ with an edge lying in $\Gamma_n$.
For any $T \in \cT_n^{\partial}$, let $I_T$ denote the open region enclosed between $T$ and $\Gamma$, so that $B_X$ may be decomposed as
\begin{equation*}
 B_X =\intt( \overline{B}_X^n \cup \bigg(\bigcup_{T \in \cT_n^{\partial}}I_T \bigg)).
\end{equation*}

Let $n \in \N$ and $u \in \VhND$.
Define $E_n u$ on $B_X^n$ by
\begin{equation}\label{eq:def-E-n-1}
 E_n u|_{B_X^n} = \begin{cases} u & \text{on }\Omega_n \\
                                0 & \text{on }U_n       \end{cases}.
\end{equation}
Furthermore, for any $T \in \cT_n^{\partial}$, define $E_n u$ on $I_T$ by
\begin{equation}\label{eq:def-E-n-2}
 E_n u (p) = u(\iota_n(p)), \qquad p \in I_n(p),
\end{equation}
where $\iota_n(p)$ is defined as the unique point on $\Gamma_n$ with the same angular coordinate as $p$.

Property (a) holds by construction. Furthermore, the map $\iota_n$ defined here coincides with $\iota_n$ from the proof of Lemma \ref{lem:iota-n}, hence property (c)
also holds.

Let $T \in \cT_n^{\partial}$.
Let $p_1$ and $p_2$ denote the two points of intersection between $I_T$ and $\Gamma$.
Since $E_n u$ is radial on $I_T$, we have
\begin{equation*}
 \abs{\nabla E_n u|_{I_T}}  = \abs*{\frac{p_1 - p_2}{|p_1 - p_2|} \cdot \nabla E_n u|_{I_T}} = \abs*{\frac{p_1 - p_2}{|p_1 - p_2|} \cdot \nabla u|_T} \leq |\nabla u|_T|.
\end{equation*}
where the second equality holds since $\nabla u$ is constant on $T$ and $(p_1 - p_2) \nabla u|_T$ only depends on $u(p_1)$ and $u(p_2)$.
Furthermore, by Taylor's theorem,
\begin{equation*}
 |u(p)| \leq |u(q)| + X |\nabla u|_T|, \qquad p \in I_T,
\end{equation*}
where $q$ is the reflection of $p$ with respect to the edge of $T$ connecting $p_1$ and $p_2$. Consequently, we have
\begin{equation}\label{eq:edge-ineq}
 \norm{E_n u}_{H^1(I_T)} \leq C_X \norm{u}_{H^1(T)}.
\end{equation}
Combining \eqref{eq:edge-ineq} with \eqref{eq:def-E-n-1}, we obtain
\begin{equation}
 \norm{E_n u }_{H^1(B_X)} \leq C_X \norm{u}_{H^1(\Omega_n)},
\end{equation}
proving the final required property (b).
\end{proof}

We have the following estimate for the boundary linear interpolation operator $\hat{\Pi}_n$.

\begin{lemma}\label{lem:interpolation}
For any $s \geq 0$, there exists a constant $C_{s,X} > 0$ such that for any $n \in \N$, we have
\begin{equation*}
 \norm{\hat \iota_n \hat \Pi_n P_n g - g}_{H^{-s}(\Gamma)} \leq C_{s,X} h^2_n N^{2 + s
 }_n\norm{g}_{H^{- s}(\Gamma)}, \qquad g \in H^{-s}(\Gamma)
\end{equation*}
and
\begin{equation*}
 \norm{\hat \Pi_n P_n g}_{H^{-s}(\Gamma_n)} \leq C_{s,X}(1 + h_n^2 N^{2 + s
 }_n) \norm{g}_{H^{-s}(\Gamma)},  \qquad g \in H^{-s}(\Gamma).
\end{equation*}

\end{lemma}

\begin{proof}
 Firstly, we may assume without loss of generality that $g \in P_n H^{-s}(\Gamma) \subset C^\infty(\Gamma)$.
 Then, by Lemma \ref{lem:iota-n} (c), it suffices to prove that
 \begin{equation*}
  \norm{(\hat \Pi_n - \hat \iota_n^{-1}) g}_{H^{-s}(\Gamma_n)}  \leq C_{s,X}h_n N^{1 + s/2
 }_n  \norm{g}_{H^{-s}(\Gamma)}.
 \end{equation*}
 We shall first estimate the left hand side as
 \begin{equation*}
  \norm{(\hat \Pi_n - \hat \iota_n^{-1}) g}_{H^{-s}(\Gamma_n)} \leq \norm{(\hat \Pi_n - \hat \iota_n^{-1}) g}_{L^2(\Gamma_n)}.
 \end{equation*}

 Let $e \subset \Gamma_n$ be any edge joining two adjacent points of contact $p_1$ and $p_2$ between $\Gamma$ and $\Gamma_n$. It suffices to prove
 \begin{equation}
 \label{eq:e-est-suffice}
  \norm{(\hat \Pi_n - \hat \iota_n^{-1}) g}_{L^2(e)} \leq
  C_{s,X}h_n N^{1 + s/2}_n  \norm{g}_{H^{-s}(\hat{e})}
 \end{equation}
 where $\hat{e} :=\iota_n^{-1}(e)) $ for the bijection $\iota_n: \Gamma \to \Gamma_n$ from the proof of Lemma
 \ref{lem:iota-n}.
 Note that $\iota_n|_{\hat{e}}$ and $\iota^{-1}_n|_{e}$ are smooth and have uniformly bounded $W^{2,\infty}$ norms.

On $e$, we may express $\hat \Pi_n g$ as
\begin{equation}
 \hat \Pi_n g(p) = g(p_1) \frac{|p - p_2|}{|p_2 - p_1|} +  g(p_2) \frac{|p - p_1|}{|p_2 - p_1|}.
\end{equation}
Let $D$ denote the derivative in the direction $p_2 - p_1$.
\begin{equation*}
 D := \frac{p_2 - p_1}{|p_2 - p_2|} \cdot \nabla
\end{equation*}
Then,
\begin{equation}
 D \hat \Pi_n g(t) = \frac{g(p_2) - g(p_1)}{|p_2 - p_1|}.
\end{equation}

Since $p_1$ lies in both $\Gamma$ and $\Gamma_n$,
we have $\hat{\iota}_n^{-1} g(p_1) = \hat{\Pi}_n g(p_1) = g(p_1)$.
Therefore,
\begin{align}\label{eq:point-est-diff}
 |(\hat{\iota}_n^{-1} - \hat{\Pi}_n ) g(p)| & \leq \int_{e} \abs*{\frac{g(p_2) - g(p_1)}{|p_2 - p_1|} - D g(\iota^{-1}_n(t))}\, \d t \nonumber \\
 & \leq h^{1/2}_n \br*{\int_{e} \abs*{\frac{g(p_2) - g(p_1)}{|p_2 - p_1|} - D g(\iota^{-1}_n(t))}^2\, \d t }^{1/2}, \qquad p \in e.
\end{align}
where in the second line we applied the Cauchy-Schwartz inequality.
Since the far right hand side of \eqref{eq:point-est-diff} does not depend on $p$, the left hand side of \eqref{eq:e-est-suffice} may be estimated as
\begin{equation}\label{eq:L2-est-1}
 \norm{(\hat \Pi_n - \hat \iota_n^{-1}) g}^2_{L^2(e)} \leq h_n^2 \int_{e} \abs*{\frac{g(p_2) - g(p_1)}{|p_2 - p_1|} - D g(\iota^{-1}_n(t))}^2\, \d t .
\end{equation}

Focus now on the integrand on the right hand side of \eqref{eq:L2-est-1}.
By the mean value theorem, there exists $\xi \in e$ such that
\begin{equation*}
 D g (\iota_n^{-1}(\xi)) = \frac{g(p_2) - g(p_1)}{|p_2 - p_1|}.
\end{equation*}
Therefore,
\begin{align}
\label{eq:integrand-est}
\abs*{\frac{g(p_2) - g(p_1)}{|p_2 - p_1|} - D g(\iota^{-1}_n(p))}& \leq \int_{\xi}^p \abs{D^2 g(\iota^{-1}_n(t))} dt \nonumber \\
&\leq h_n^{1/2} \br*{\int_e \abs{D^2 g(\iota^{-1}_n(t))}^2 dt }^{1/2}, \qquad p \in e.
\end{align}
Substituting \eqref{eq:integrand-est} into \eqref{eq:L2-est-1} and using the fact that
the far right hand side of \eqref{eq:integrand-est} is independent of $p$, we obtain
\begin{align*}
 \norm{(\hat \Pi_n - \hat \iota_n^{-1}) g}^2_{L^2(e)}
  & \leq h_n^4 \int_e \abs{D^2 g(\iota^{-1}_n(t))}^2 dt \\
  & \leq h_n^4 \norm{\iota_n}_{W^{2,\infty}(\hat e)} \norm{\iota^{-1}_n}^2_{W^{2,\infty}(\hat e)} \norm{g}_{H^2(\Gamma)} \\
  & \leq C_{X} h_n^4 N_n^{4 +2s} \norm{g}^2_{H^{-s}(\Gamma)}
\end{align*}
as required, where the last line holds by the properties of $\iota_n$ already mentioned and the fact that $e'_\alpha = i \alpha e_\alpha$.
\end{proof}

Finally, we may  deduce the following uniform trace operator estimate directly from the properties of the operators $\hat \iota_n$ and $E_n$ proved above.

\begin{lemma}\label{lem:trace-uniform}
For any $s > \f12$, the trace operators $\gamma_{\Gamma_n}$, $n \in \N$, on $\Gamma_n$ satisfy
\begin{equation*}
 \norm{\gamma_{\Gamma_n} u }_{H^{s - \f12}(\Gamma_n)} \leq  C_{s,X} \norm{u}^s_{H^1(\Omega_n)} \norm{u}^{1 - s}_{L^2(\Omega_n)}, \qquad u \in H^1(\Omega_n).
\end{equation*}
\end{lemma}
\begin{proof}

By construction, we have
\begin{equation}\label{eq:trace-Gam-n}
 \gamma_{\Gamma_n} = \hat{\iota}_n^{-1} \gamma_\Gamma E_n, \qquad n \in \N.
\end{equation}
Consequently,
\begin{align*}
  \norm{\gamma_{\Gamma_n} u }_{H^{s - \f12}(\Gamma_n)}
 & \leq \norm{\hat{\iota}_n^{-1}}_{H^{s-\f12}(\Gamma) \to H^{s-\f12}(\Gamma_n)}\norm{\gamma_\Gamma}_{H^s(B_X) \to H^{s- \f12}(\Gamma)} \norm{E_nu}_{H^s(\Omega_n)} \\
 & \leq C_{s,X} \norm{E_n u}_{H^s(\Omega_n)} \\
 & \leq C_{s,X} \norm{E_n u}^s_{H^1(\Omega_n)} \norm{E_n u}^{1-s}_{L^2(\Omega_n)} \\
 & \leq C_{s,X} \norm{u}^s_{H^1(\Omega_n)} \norm{u}^{1 - s}_{L^2(\Omega_n)},
\end{align*}
where:
\begin{itemize}
 \item the second inequality holds by Lemma \ref{lem:iota-n} (c) and the usual trace theorem \cite[Th. 3.37]{mclean2000strongly},
 \item the third inequality holds by a Sobolev interpolation theorem \cite{BrezisInterpolation} and
 \item the fourth line holds by Lemma \ref{lem:E-n} (b).
\end{itemize}

The extension operator $E_n$ is uniformly bounded from $H^1 \to H^1$ by Lemma \ref{lem:E-n} (b), the operator $\hat{\iota}_n^{-1}$ is uniformly bounded from $H^{\f12} \to H^{\f12}$ by Lemma \ref{lem:iota-n} (c) and $\gamma_\Gamma$ is uniformly bounded from $H^1 \to H^{\f12}$ by the usual trace theorem,
hence the lemma follows.
\end{proof}

\subsection{Matrix and sum truncation errors}\label{subsec:mat-sum}
First, we show convergence of the matrix truncation error, which follows easily from the properties of $K(k)$.

\begin{lemma}\label{lemma:pointwise-to-uniform}
	Let $A\subset\C$ be open and simply connected. Let $(K_n)_{n\in\N}$ be a sequence of analytic, operator-valued functions on $A$ converging pointwise to 0. Then for any compact subset $C\subset A$ one has 
	\begin{align*}
		\sup_{k\in C} \|K_n(k)\| \to 0 \qquad \text{as} \qquad n \to \infty.
	\end{align*}
\end{lemma}
\begin{proof}
	Let $C\subset A$ be compact and choose a smooth closed curve $\gamma$ in $A$, which encloses $C$ with winding number 1. Moreover, assume that $\gamma$ is chosen such that $\delta := \inf_{z\in\gamma} \dist(z,C)>0$. By Cauchy's integral formula we have for any $w\in C$
	\begin{align*}
		K_n(w) &= \f{1}{2\pi \i}\oint_\gamma \f{K_n(z)}{z-w}\,dz
	\end{align*}
	and hence
	\begin{align}\label{eq:caychy_integral_consequence}
		\| K_n(w)\| &\leq \f{1}{2\pi\delta}\oint_\gamma \|K_n(z)\| \,dz.
	\end{align}
	By dominated convergence the integral on the right-hand side of \eqref{eq:caychy_integral_consequence} converges to 0 as $n\to\infty$. Since the right-hand side of \eqref{eq:caychy_integral_consequence} is independent of $w$ we conclude that
	\begin{align*}
		\sup_{w\in C}\| K_n(w)\| &\leq \f{1}{2\pi\delta}\oint_\gamma \|K_n(z)\| \,dz \qquad \text{as} \qquad n \to \infty.
	\end{align*}
\end{proof}

\begin{lemma}\label{lem:E-mat-conv}
    For each $k \in \C_-$, the matrix truncation error satisfies
    \begin{equation*}
		\cE_{\mathrm{mat}}^n(k) \to 0
    \end{equation*}
    locally uniformly in $\C_-$.
\end{lemma}
\begin{proof}
	Let $k\in\C_-$. Adding and subtracting $K(k)P_n$, we have
	\begin{align*}
		K(k) - P_nK(k)P_n = K(k)(I - P_n) + (I - P_n)K(k)P_n
	\end{align*}
	By Lemma \ref{lem:K-prop}, we have that $K(k)$ is compact and uniformly bounded in $\C_-$. For fixed $k\in\C_-$, by compactness of $K(k)$ and strong convergence of $I - P_n$ to 0 it follows that 
	\begin{align*}
		K(k)(I - P_n) &\to 0
		\\
		(I - P_n)K(k)P_n &\to 0
	\end{align*}
	in operator norm. Thus $\|K(k) - P_nK(k)P_n\| \to 0$ pointwise in $\C_-$.
	The desired local uniform convergence now immediately follows from Lemma \ref{lemma:pointwise-to-uniform}.
\end{proof}

Next, we focus on the sum truncation error. First we perform an eigenfunction expansion for the discretisation $\hMinn(k)$. Recall that $d_n = \dim(\VhND)$.

\begin{lemma}
\label{lem:expansion-for-S_n}
	Then for any $g \in P_nH^{-\f12}(\Gamma)$ one has
	\begin{align}\label{eq:expansion_for_S_n}
		\hMinn(k)g = \sum_{m=1}^{d_n} \frac{1}{\mu_m^n - k^2} (\hat \Pi_n  g, \gamma_{\Gamma_n}w_m^n)_{L^2(\Gamma_n)} \gamma_{\Gamma_n} w_m^n
	\end{align}
	where $\mu_m^n$ and $w_m^n$ are the FEM approximations for the eigenvalues and eigenfunctions of $- \DND$ (cf. \eqref{eq:fem-eigs}).

\end{lemma}
\begin{proof}
	Let $g \in P_nH^{-\f12}(\Gamma)$ and $u := S_n(k)g$. The eigenfunctions $w_m^n$ form a basis for $\VhND$ and thus
	\begin{align}\label{eq:basis_expansion_S_n}
		u = \sum_{m=1}^{d_n} \inner{u,w_m^n}_{L^2(\Omega_n)}w_m^n.
	\end{align}
	By \eqref{eq:S-n-defn} and \eqref{eq:fem-eigs} we have
	\begin{align}
		\inner{\nabla u,\nabla w_m^n}_{L^2(\Omega_n)} - k^2\inner{u,w_m^n}_{L^2(\Omega_n)} &= (\hat \Pi_n g,\gamma_{\Gamma_n}w_m^n)_{L^2(\Gamma_n)}
		\label{eq:galerkin_Sn_with_eigenfunction}
		\\
		\inner{\nabla w_m^n,\nabla u}_{L^2(\Omega_n)} &= \mu_m^n \inner{w_m^n,u}_{L^2(\Omega_n)}
		\nonumber
	\end{align}
	and since the $\mu_m^n$ are real
	\begin{align}\label{eq:u_U_mu}
		\inner{\nabla u,\nabla w_m^n}_{L^2(\Omega_n)} &= \mu_m^n \inner{u,w_m^n}_{L^2(\Omega_n)}
	\end{align}
	Substituting \eqref{eq:u_U_mu} into \eqref{eq:galerkin_Sn_with_eigenfunction}, we get
	\begin{align}\label{eq:expression_for_uw_mn}
		\inner{u,w_m^n}_{L^2(\Omega_n)} = \f{1}{\mu_m^n - k^2}(\hat \Pi_n  g,\gamma_{\Gamma_n}w_m^n)_{L^2(\Gamma_n)}.
	\end{align}
	The proof is completed by substituting \eqref{eq:expression_for_uw_mn} into \eqref{eq:basis_expansion_S_n}.
\end{proof}

The next lemma shows convergence of the sum truncation error.
Our strategy shall be to  utilise Weyl's law in conjunction with the min-max principal to get a lower bound for the eigenvalues $\mu_m^n$.
\begin{prop}\label{prop:E-sum}
For any $s \in (0,1)$ and $k \in \C_-$, the sum truncation error satisfies
\begin{equation}
\label{eq:Esum-est}
 \Esum(k) \leq C_{s,X,k_0}(k) \br*{1 + h_n^2 N_n^{\f52}}^2 N_n^s \abs*{J_n^{-\frac{s}{2}} - d_n^{-\frac{s}{2}}}
\end{equation}
for some locally uniform constant $C_{s,X,k_0}(k)> 0$ independent of $n$. Consequently, if the limit $\lim_{n \to \infty} N_n J_n^{-\f12} = 0$ holds and the sequence $(h_n N^{\f54}_n)_{n\in \N}$ is bounded, then
\begin{equation*}
 \Esum(k) \to 0 \quad \text{as} \quad n \to \infty, \qquad k \in \C_-.
\end{equation*}
\end{prop}
\begin{proof}
If suffices to show that
\begin{equation*}
\abs*{\br{P_n(\hat \Pi_n^* \hMinn(k) - \Min^n(k))P_n g_1, g_2}_{L^2(\Gamma)}} \leq \cC_n \norm{g_1}_{H^{-\f12}(\Gamma)}\norm{g_2}_{H^{-\f12}(\Gamma)}
\end{equation*}
for all $g_1, g_2 \in H^{-\f12}(\Gamma)$, where $\cC_n > 0$ denotes the right hand side of \eqref{eq:Esum-est}.
using the definition of $\Min^n(k)$ (see \eqref{eq:a-n-matrix-elements} and \eqref{eq:Minn-defn}), we have
\begin{align*}
 \br{P_n(&\hat \Pi_n^* \hMinn(k)  - \Min^n(k))P_n g_1, g_2}_{L^2(\Gamma)}  = \br{ \hMinn(k)P_n g_1, \hat \Pi_n P_n g_2}_{L^2(\Gamma_n)} - \br{\Min^n(k)P_n g_1, P_n\hat  g_2}_{L^2(\Gamma)}   \\
 & = \br{(\hMinn(k) - \hMinn(k_0))P_n g_1, P_n\hat \Pi_n g_2}_{L^2(\Gamma_n)} -  \sum_{m=1}^{J_n} \frac{k^2 - k_0^2}{(\mu^n_m - k^2)(\mu^n_m - k_0^2)}(\hat \Pi_n e_\alpha,\gamma_{\Gamma_n} w^n_m)_{L^2(\Gamma_n)}\gamma_{\Gamma_n} w^n_m.
\end{align*}
Using the eigenfunction expansion for $\hMinn(k)$ from Lemma \ref{lem:expansion-for-S_n}, we have
\begin{equation*}
 (\hMinn(k) - \hMinn(k_0))P_n g_1 = \sum_{m=1}^{d_n} \frac{k^2 - k_0^2}{(\mu^n_m - k^2)(\mu^n_m - k_0^2)}(\hat \Pi_n P_n g_1,\gamma_{\Gamma_n} w^n_m)_{L^2(\Gamma_n)}\gamma_{\Gamma_n} w^n_m.
\end{equation*}
Consequently, we have
\begin{multline}\label{eq:corrector}
\br{P_n(\hat \Pi_n^* \hMinn(k)  - \Min^n(k))P_n g_1, g_2}_{L^2(\Gamma)} \\ = \sum_{m=J_n+1}^{d_n} \frac{k^2 - k_0^2}{(\mu^n_m - k^2)(\mu^n_m - k_0^2)}(\hat \Pi_n P_n g_1,\gamma_{\Gamma_n} w^n_m)_{L^2(\Gamma_n)}(\gamma_{\Gamma_n} w^n_m,\hat \Pi_n P_n g_2)_{L^2(\Gamma_n)}.
\end{multline}

The remainder of the proof consists in estimating the sum on the right hand side of \eqref{eq:corrector}. Focus first on the inner products and estimate using duality pairing,
\begin{equation}\label{eq:inner-in-sum}
 \abs*{(\hat \Pi_n P_n g_1,\gamma_{\Gamma_n} w^n_m)_{L^2(\Gamma_n)}} \leq  \norm{\hat \Pi_n P_n g_1}_{H^{-\f12+\frac{s}{2}}(\Gamma_n)} \norm{\gamma_{\Gamma_n} w^n_m}_{H^{\f12 - \frac{s}{2}}(\Gamma_n)}
\end{equation}
By  Lemma \ref{lem:interpolation}, we have
\begin{equation*}
 \norm{\hat \Pi_n P_n g_1}_{H^{-\f12+\frac{s}{2}}(\Gamma_n)} \leq C_{X,s} (1 + h_n^2 N_n^{\f52} )
 \norm{P_n g_1}_{H^{-\f12 + \frac{s}{2}}(\Gamma)}
 \leq C_{s,X}(1 +h_n^2 N_n^{\f52})
  N_n^{\frac{s}{2}} \norm{g_1}_{H^{-\f12}(\Gamma)}.
\end{equation*}
Furthermore, using the uniform trace estimate Lemma \ref{lem:trace-uniform} and the equation \eqref{eq:fem-eigs} for $(\mu_m^n,w_m^n)$, we have
\begin{equation*}
 \norm{\gamma_{\Gamma_n} w^n_m}_{H^{\f12 - \frac{s}{2}}(\Gamma_n)}
 \leq C_{s,X} \norm{ w^n_m}_{H^{1}(\Omega_n)}^{1 - \frac{s}{2}} \norm{ w^n_m}_{L^2(\Omega_n)}^{ \frac{s}{2}} = C_{s,X} \abs{1 + \mu_m^n}^{\frac{1}{2} - \frac{s}{4}}.
\end{equation*}
Substituting back into \eqref{eq:inner-in-sum}, we obtain the estimate
\begin{equation}
\label{eq:inner-in-sum-2}
\abs*{(\hat \Pi_n P_n g_1,\gamma_{\Gamma_n} w^n_m)_{L^2(\Gamma_n)}} \leq C_{s,X} \abs{1 + \mu_m^n}^{\frac{1}{2} - \frac{s}{4}} (1 + h_n^2 N_n^{\f52})
  N_n^{\frac{s}{2}} \norm{g_1}_{H^{-\f12}(\Gamma)}.
\end{equation}
A similar estimate holds for the other inner product on the right hand side of   \eqref{eq:corrector}.

In addition, $k$ and $k_0$ have non-zero imaginary part hence $|\mu_m^n - k^2| \geq C(k) |1 + \mu_m^n|$.
Using this, along with estimate \eqref{eq:inner-in-sum-2} and the formula \eqref{eq:corrector}, we obtain
\begin{equation}\label{eq:corrector-est}
 |\br{P_n(\hat \Pi_n^* \hMinn(k)  - \Min^n(k))P_n g_1, g_2}_{L^2(\Gamma)}| \leq C_{s,X,k_0}(k) (1 + h_n^2 N_n^{\f52})^2
  N_n^{s} \sum_{j=J_n+1}^{d_n}\frac{1}{|1 + \mu_m^n|^{1 + \frac{s}{2}}}.
\end{equation}
Next, observe that
\begin{equation*}
 \VhND \subset \HND \subseteq H^1(B_X).
\end{equation*}
Consequently, by the min-max principal and the Weyl law for the Neumann Laplacian on the ball, we have
\begin{equation}\label{eq:eig-lower}
 \mu_m^n \geq \mu_m(\Omega) \geq \mu_m(B_X) \geq C_X m.
\end{equation}
The proof is completed by substituting \eqref{eq:eig-lower} into \eqref{eq:corrector-est} and bounding the sum by an appropriate integral.
\end{proof}

\subsection{Finite element error}\label{subsec:mosco-fem}
In this subsection, we prove convergence of the the finite element solution operator $S_n(k)$, which shall later be used to prove convergence of $\Esol(k)$.
First, we establish Mosco convergence of the finite element space $\VhND$.

\begin{lemma}\label{lem:fem-mosco}
 If $H^1_0(U_n^c) \Mto H^1_0(U^c)$ as $n \to \infty$, then
 \begin{equation*}
  E_n \VhND \Mto \HND \quad \text{to} \quad n \to \infty.
 \end{equation*}
\end{lemma}
\begin{proof} Let $\cH = H^1(B_X)$ and  $\tilde{\cW} := \set{\phi|_{\Om}: \phi \in C^\infty_c(U^c)} \subset \HND$.
By Lemma \ref{lem:W_n+}, it suffices to show that the following.
\begin{enumerate}[label=(\roman*)]
 \item For every $u \in \tilde{\cW}$, there exists a sequence $u_n \in E_n \VhND$, $n \in \N$, with $\norm{u - u_n}_{H^1(B_X)} \to 0$ as $n \to \infty$.
 \item We have $\cW^+ := \set{u|_{B_X}: u \in H^1_0(U_n^c)} \Mto \HND$ as $n \to \infty$.
\end{enumerate}

Focusing on (i), let $u \in \tilde \cW$. Let $u_n = E_n \Pi_n u$, $n \in \N$,
where the interpolation is defined since $u$ may be regarded as a function in $C(\Om_n)$ via extension by zero. Then, we have
\begin{align}\label{eq:H1-lim-lem-mosco}
 \norm{u - u_n}_{H^1(\Om)} & \leq \norm{u}_{H^1(\Om \bs \Om_n)} + \norm{u_n}_{H^1(\Om \bs \Om_n)} + \norm{u - \Pi_n u}_{H^1(\Om_n)} \nonumber\\
 & \leq \norm{u}_{H^1(\Om \bs \Om_n)} + \norm{u_n}^{\f12}_{H^3(\Om \bs \Om_n )} |\Om \bs \Om_n|^{\f12} + C_\theta h_n \norm{u}_{H^2(B_X)}
\end{align}
where in the first inequality we used the fact that $u_n  = \Pi_n u $ on $\Om_n$ and in the second inequality we used H\"older's and Morrey's inequalities
(for the second term), as well as a standard interpolation inequality (for the third term).
Furthermore, the sequence $\norm{u_n}_{H^3(\Om \bs \Om_n )}$, $n \in \N$, is bounded by the boundedness property of the extension operators $E_n$, $n \in \N$, (Lemma \ref{lem:E-n} (b)) and interpolation inequalities.
Therefore, observe that each term on the right hand side of \eqref{eq:H1-lim-lem-mosco} tends to zero as $n \to \infty$, establishing (ii).

To show (ii), we need to verify the two hypotheses of Mosco convergence in Definition \ref{def:mosco}. We shall utilise a bounded extension operator $E: H^1(B_X) \to H^1(\R^2)$.

Firstly, for any $u \in \HND$, we need to show that there exists a sequence $u_n \in \cW^+$,
$n \in \N$, such that $u_n \to u$ in $H^1(B_X)$. By Mosco convergence of $(H^1_0(U_n^c))_{n \in \N}$, and since we have since $E u \in H^1_0(U^c)$, there exists a sequence $v_n \in H^1_0(U_n^c)$, $n \in \N$, such that $v_n \to E u$ in
$H^1(\R^2)$. We obtain the desired sequence by letting $u_n := v_n|_{B_X}$

Secondly, let $u_{n_j} \in V^{n_j}_{\mathrm{ND}}(\Om_{n_j})$, $j \in \N$, be some subsequence such that $u_{n_j} \wto u $ as $j \to \infty$ in $H^1(B_X)$ for some
$u \in H^1(B_X)$. We need to show that $u \in \HND$.
Since strong-strong continuity implies weak-weak continuity, we have that $E E_{n_j} u_{n_j} \wto E u$ as $j \to \infty$ in $H^1(\R^2)$.
By Mosco convergence of $(H^1_0(U_n^c))_{n \in \N}$, we have that $E u \in H^1_0(U^c)$ hence $u \in \HND$ as required.

\end{proof}

The convergence result for $S_n(k)$ reads as follows.
\begin{prop}\label{prop:mosco-to-conv}
 If $H^1_0(U_n^c) \Mto H^1_0(U^c)$ as $n \to \infty$ and $\lim_{n \to \infty} h_n N_n = 0$, then it holds that
  \begin{equation*}
   \norm{E_n S_n(k)P_n - S(k)}_{H^{-\f12} \to H^1} \to 0 \quad \text{as} \quad n \to \infty, \qquad k \in \C_-.
  \end{equation*}
  Furthermore, for any  $g \in H^{-\f12}(\Gamma)$, it holds that
  \begin{equation}
   \norm{S(k)P_n g}_{H^{1}(\Omega_n)} \leq C_X(k) \norm{g}_{H^{-\f12}(\Gamma)}, \qquad k \in \C_-,
  \end{equation}
  where the constant $C_X(k) >0$ is locally uniform in $\C_-$.
\end{prop}
\begin{proof} Let  $k \in \C_-$, let $(g_n)_{n \in \N} \subset H^{-\f12}(\Om)$ be any bounded sequence. Consider the sequence of solutions $u_n \in \VhND$, $n \in \N$, of the FEM problems
\begin{equation}\label{eq:fem-f-g}
 \inner{\nabla u_n, \nabla \phi_n}_{L^2(\Om_n)} - k^2 \inner{u_n, \phi_n}_{L^2(\Om_n)} = (\hat \Pi_n P_n g_n, \phi_n|_{\Gamma_n })_{L^2(\Gamma_n)}, \qquad \forall \phi_n \in \VhND.
\end{equation}
Observe that we have $u_n = S_n(k) P_n g_n$.

By setting $\phi = u_n $ and considering the real and imaginary parts of \eqref{eq:fem-f-g}, one may see that
\begin{align*}
 \norm{u_n}^2_{H^1(\Om_n)}  \leq C(k) \abs*{(\hat \Pi_n P_n g, u_n |_{\Gamma_n})_{L^2(\Gamma_n)}} & \leq C_X(k) \norm{\hat \Pi_n P_n g_n}_{H^{-\f12}(\Gamma_n)}\norm{u_n}_{H^1(\Om_n)} \\
 & \leq C_X(k) \norm{g_n}_{H^{-\f12}(\Gamma)} \norm{u_n}_{H^1(\Om_n)}
\end{align*}
where in the second inequality we used the uniform boundedness
of trace operator $\gamma_{\Gamma_n}$ (Lemma \ref{lem:trace-uniform}),
as well as Cauchy-Schwartz for the duality pairing,
and in the third inequality we used Lemma \ref{lem:interpolation}.
Consequently, the sequence $(u_n)_{n \in \N}$ is also bounded in $H^1$.
Notice that, since $u_n = S_n(k) P_n g_n$, this already proved the second statement of the proposition.
Furthermore, by uniform boundedness of $E_n$ (Lemma \ref{lem:E-n}),
the sequence $(E_n u_n)_{n \in \N}$ is bounded in $H^1(B_X)$.

By weak compactness, there exists a subsequence $(n_j)_{j \in \N} \subset \N $, such that
\begin{equation*}
E_{n_j} u_{n_j} \wto u \quad \text{as} \quad j \to \infty \quad \text{in} \quad
H^1(B_X) \quad
\text{for some} \quad u \in H^1(B_X).
\end{equation*}
By Lemma \ref{lem:fem-mosco}, it in fact holds that $u \in \HND$.
Fix any $\phi \in \HND$.
By Lemma \ref{lem:fem-mosco}, there exists a sequence $\phi_n \in \VhND$,
$n \in \N$, such that $E_n \phi_n \to \phi$ as $n \to \infty$ in $H^1(B_X)$.
Furthermore, by weak compactness of $H^{\f12}(\Gamma)$, there exists a subsequence $(g_{n_j})_{j \in \N}$ such that
$\hat \Pi_{n_j} P_{n_j} g_{n_j} \wto g$ as $j \to \infty$ in $H^{-\f12}(\Gamma)$ for some $g \in H^{-\f12}(\Gamma)$.
The remainder of the proof consists in computing limits for each of the inner products in the Galerkin equation \eqref{eq:fem-f-g} (for the subsequence $n_j$).
For simplicity, we rename $n_j \mapsto j$.

Focus first on the left most inner product in \eqref{eq:fem-f-g}.
By adding and subtracting the appropriate term
and using the fact that $E_j u_j = u_j$ on $\Om_j$, we have
\begin{align}
 \big|\inner{\nabla u_j , \nabla \phi_j}_{L^2(\Om_j)}  & - \inner{\nabla u, \nabla \phi}_{L^2(\Om)}\big|
 \nonumber\\ & \leq
 \underbrace{\abs*{\inner{\nabla E_j u_j , \nabla (E_j \phi_j - \phi)}_{L^2(\Om_j)}}}_{(T1)}
 + \underbrace{\abs*{\inner{\nabla (E_j u_j - u), \nabla \phi}_{L^2(\Om_j)}}}_{(T2)}
 + \underbrace{\abs*{\inner{\nabla u , \nabla \phi}_{L^2(\Om \bs \Om_j)}}}_{(T3)}.
\end{align}
The terms (T1) and (T3) tend to zero as $j \to \infty$, by strong convergence
and continuity of measure respectively. The term (T2) is further estimated as
\begin{equation}\label{eq:bound-term-T2}
\abs*{\inner{\nabla (E_j u_j - u), \nabla \phi}_{L^2(\Om_j)}} \leq \abs*{\inner{\nabla (E_j u_j - u), \nabla \phi}_{L^2(\Om)}} + \br*{\norm{E_j u_j}_{H^1(\Om)} + \norm{u}_{H^1(\Om)}}\norm{\phi}_{H^1(\Om\bs \Om_j)}
\end{equation}
The first term on the right hand side of \eqref{eq:bound-term-T2} tends to zero by weak convergence whereas the second tends to zero by continuity of measure and the boundedness of $(E_n u_n)$ in $H^1$. Consequently, the term  (T2) also tends to 0 as $j \to \infty$, hence
\begin{equation}\label{eq:first-inner-prod}
 \inner{\nabla u_j , \nabla \phi_j}_{L^2(\Om_j)} \to \inner{\nabla u, \nabla \phi}_{L^2(\Om )}\quad \text{as} \quad j \to \infty.
\end{equation}
We can similarly show that
\begin{equation}\label{eq:second-inner-prod}
\inner{ u_j , \phi_j}_{L^2(\Om_j)} \to \inner{ u, \phi}_{L^2(\Om )}\quad \text{as} \quad j \to \infty.
\end{equation}

Focus on the right hand side of \eqref{eq:fem-f-g}.
Using the properties of $\hat \iota_n$ (Lemma \ref{lem:iota-n}), we estimate as
\begin{align}\label{eq:duality-big-est}
 \Big|(\hat \Pi_j P_j g_j ,\phi_j|_{\Gamma_j})_{L^2(\Gamma_j)}  &- (g_j, \phi|_{\Gamma})_{H^{-\f12},H^{\f12}}\Big| \nonumber\\
 &  \leq  \underbrace{\abs*{(\hat \iota_j \hat \Pi_j P_j g_j , \hat \iota_j \phi_j|_{\Gamma_j} - \phi|_\Gamma)_{L^2(\Gamma)}}}_{(T4)} + \underbrace{\abs*{(\hat \iota_j \hat \Pi_j P_j g_j  - g, \phi|_\Gamma)_{H^{-\f12},H^{\f12}}}}_{(T5)}.
\end{align}
The term (T4) is further estimated as
\begin{align}\label{eq:mosco-T4}
 \big |(\hat \iota_j \hat \Pi_j P_j g_j , \hat \iota_j \phi_j|_{\Gamma_j} &- \phi|_\Gamma)_{L^2(\Gamma)} \big | \nonumber\\
 & \leq \norm{\hat \iota_j \hat \Pi_j P_j g_j}_{H^{-\f12}(\Gamma)}
 \br*{\norm{\hat{\iota}_j \phi_j|_{\Gamma_j} - (E_j \phi_j)|_{\Gamma}}_{H^{\f12}(\Gamma)} + \norm{(E_j \phi_j)|_{\Gamma} - \phi|_\Gamma}_{H^{\f12}(\Gamma)}}.
\end{align}
The right hand side of \eqref{eq:mosco-T4} tends to zero as $j \to \infty$ by Lemma \ref{lem:E-n} (c), the interpolation estimate Lemma \ref{lem:interpolation} and strong convergence of $\phi_j$.
Furthermore, by the strong convergence of $\hat \iota_j \hat \Pi_j P_j$ (Lemma \ref{lem:interpolation}), the sequence
 $\hat \iota_j \hat \Pi_j P_j g_j$ tends weakly to $g$ in $H^{-\f12}(\Gamma)$, showing that (T5) tends to zero.
Consequently, we have
\begin{equation}\label{eq:fourth-inner-prod}
 (\hat \Pi_j P_j g ,\phi_j|_{\Gamma_j})_{L^2(\Gamma_j)}  \to (g, \phi|_{\Gamma})_{H^{-\f12},H^{\f12}} \quad \text{as} \quad j \to \infty.
\end{equation}

Since $\phi \in \HND$ above was arbitrary, the limits \eqref{eq:first-inner-prod}, \eqref{eq:second-inner-prod} and \eqref{eq:fourth-inner-prod} show that $u \in \HND$ satisfies the variational equation
\begin{equation}\label{eq:var-f-g}
 \inner{\nabla u, \nabla \phi}_{L^2(\Om)} - k^2 \inner{u, \phi}_{L^2(\Om)} = (g, \phi|_{\Gamma})_{H^{-\f12},H^{\f12}}, \qquad \forall \phi \in \HND,
\end{equation}
that is $u = S(k) g$.

The above argument may be repeated for any weakly converging subsequence of $(u_n)_{n \in \N}$, hence  we conclude that
$E_n u_n \wto u$ as $n \to \infty$ in $H^1(B_X)$.
By Rellich's theorem, this implies that $E_n u_n \to u$ as $n \to \infty$ in $L^2(B_X)$.
Furthermore, setting $\phi = u_n$ in \eqref{eq:fem-f-g} and taking the limit $n \to \infty$ shows that $\norm{\nabla u_n}_{L^2(\Omega_n)} \to \norm{\nabla u}_{L^2(\Omega)}$, hence we also have strong convergence $u_n \to u$ in $H^1(B_X)$ as $n \to \infty$.
The proposition follows since we have show that
\begin{equation*}
 E_n S_n(k) P_n g  \to S(k)g \quad \text{as} \quad n \to \infty \qquad
 \text{for some} \qquad g \in H^{-\f12}(\Gamma).
\end{equation*}
\end{proof}

\subsection{Proof of Proposition \ref{prop:Tn}}
\label{subsec:proof}

First we focus on proving (a).
It was proven that $\Emat(k) \to 0$ and $\Esum(k) \to 0$ in Lemma \ref{lem:E-mat-conv} and Proposition \ref{prop:E-sum} respectively so, by Lemma \ref{lem:initial}, it suffices to prove that $\Esol(k) \to 0$ as $n \to \infty$ for all $k \in \C_-$.
Focusing on $\Esol(k)$, we estimate as
\begin{align}\label{eq:Esol-split}
 \Esol(k) & \leq \norm{P_n(\Min(k) - \hat{\Pi}_n^*\hMinn(k)P_n)}_{H^{-\f12} \to H^{\f12}} \norm{P_n}_{H^{-\f12} \to H^{-\f12}} \nonumber \\
 & \leq \underbrace{\norm{P_n(\Min(k) - \hat{\iota}_n\hMinn(k)P_n)}_{H^{-\f12} \to H^{\f12}}}_{(T1)} + \underbrace{\norm{P_n (\hat{\iota}_n - \hat{\Pi}_n^*)\hMinn(k)P_n}_{H^{-\f12} \to H^{\f12}}}_{(T2)}.
\end{align}
where we used the fact that $\norm{P_n}_{H^{-\f12} \to H^{-\f12}} = 1$.
Focusing on the term $(T1)$ and using Lemma \ref{lem:E-n} (c), the trace theorem and Proposition \ref{prop:mosco-to-conv}, we have
\begin{align}
 (T1) \leq \norm{\Min(k) - \hat{\iota}_n\hMinn(k)P_n}_{H^{-\f12} \to H^{\f12}}
 & \leq \norm{\gamma_\Gamma(S(k) - E_n S_n(k)P_n)}_{H^{-\f12} \to H^{\f12}} \nonumber \\ & \leq \norm{S(k) - E_n S_n(k)P_n}_{H^{-\f12} \to H^{1}} \to 0 \qquad \text{as} \qquad n \to \infty.
\end{align}
To see that the term $(T2)$ tends to zero, let $\phi, \psi \in H^{-\f12}(\Gamma)$.
Then,
\begin{align}
\Big| \Big( P_n( \hat \Pi_n^* & -  \hat \iota_n ) \hMinn(k)  P_n \phi,  \psi\Big)_{H^{-\f12},H^\f12} \Big|
 = \abs*{\br*{\hMinn(k)P_n \phi, \hat \Pi_n P_n\psi}_{H^{-\f12},H^\f12} - \br*{ \hat \iota_n \hMinn(k)P_n \phi, P_n \psi}_{H^{-\f12},H^\f12}} \nonumber\\
 & \leq \abs*{\br*{\hat \iota_n \hMinn(k)P_n \phi, \hat \iota_n  \hat \Pi_n P_n\psi}_{H^{-\f12},H^\f12} - \br*{\hMinn(k)P_n \phi, \hat \Pi_n P_n\psi}_{H^{-\f12},H^\f12}}
  \nonumber \\
  & \hspace{220pt} + \abs*{\br*{ \hat \iota_n \hMinn(k)P_n \phi,  \hat \iota_n  \hat \Pi_n P_n \psi -  P_n \psi}_{H^{-\f12},H^\f12}} \nonumber \\
  & \leq \delta_X(n) \norm{\hMinn(k)P_n \phi}_{H^{\f12}(\Gamma_n)} \norm{ \hat \Pi_n  P_n \psi }_{H^{-\f12}(\Gamma_n )} + \norm{\hat \iota_n\hMinn(k)P_n \phi}_{H^{\f12}(\Gamma)} \norm{(\hat \iota_n \hat \Pi_n P_n - 1)   P_n \psi}_{H^{-\f12}(\Gamma)} \nonumber \\
  & \leq C_X(k) \br*{\delta_X(n) + h_n^2 N_n^{\f52}} \norm{\phi}_{H^{-\f12}(\Gamma)} \norm{\psi}_{H^{-\f12}(\Gamma)},
\end{align}
where in the third line we used Lemma \ref{lem:iota-n} (b) and in the fourth line we used Lemma \ref{lem:interpolation} (to estimate the $H^{-\f12}$ norms)
as well as Lemma \ref{lem:trace-uniform}, Lemma \ref{lem:E-n} (c) and Proposition \ref{prop:mosco-to-conv} (to estimate  the $H^{\f12}$ norms).
The factor $\delta_X(n) + h_n^2 N_n^{\f52}$  tends to zero as $n \to \infty$ hence  the term $(T2)$ also tends to zero as $n \to \infty$, completing the proof of (a).

Next, we prove (b); it suffices to prove that $\sup_{n \in \N}\norm{\tilde T_n(k)}_{L^2 \to L^2} < \infty$ for every $k \in \C_-$.
Firstly, we have
\begin{align*}
 \norm{\tilde T_n(k)}_{L^2 \to L^2} \leq 1 + \norm{T_n(k)}_{L^2 \to L^2}
 & \leq 1 + \norm{P_n\cN^{\f12} (N_1(k)+ \Min^n(k) N_2(k)) \cN^{-\f12}P_n}_{L^2 \to L^2} \\
 & \leq C_X(k) ( 1 + \norm{P_n \Min^n(k)P_n}_{H^{-\f12} \to H^{\f12}})
\end{align*}
for some locally uniform constant $C_X(k) > 0$,
where the final inequality holds by the locally uniform boundedness of $N_1(k)$ and $N_2(k)$ between their respective spaces.
We therefore focus on proving boundedness of $\norm{P_n \Min^n(k)P_n}_{H^{-\f12} \to H^{\f12}}$.
Adding and subtracting by $\hat \Pi_n^* \hMinn(k)$, we have
\begin{align}
\label{eq:final-ineq}
 \norm{P_n \Min^n(k)P_n}_{H^{-\f12} \to H^{\f12}} & \leq \norm{P_n (\Min^n(k) - \hat \Pi_n^* \hMinn(k))P_n}_{H^{-\f12} \to H^{\f12}}  + \norm{P_n  \hat \Pi_n^* \hMinn(k)P_n}_{H^{-\f12} \to H^{\f12}}\nonumber \\
 & \leq C_{X,k_0}(k) + \norm{P_n  \hat \Pi_n^* \hMinn(k)P_n}_{H^{-\f12} \to H^{\f12}},
\end{align}
where we used  Proposition \ref{prop:E-sum} in the second inequality.
It therefore suffices to prove boundedness of the second term on the right hand side of \eqref{eq:final-ineq}.

Let $\phi, \psi \in H^{-\f12}(\Gamma)$. Then,
\begin{align*}
 \abs*{\br*{P_n  \hat \Pi_n^* \hMinn(k)P_n \phi,\psi}_{H^{-\f12},H^{\f12}}} & = \abs*{\br*{\hMinn(k)P_n \phi, P_n  \hat \Pi_n\psi}_{H^{-\f12},H^{\f12}}} \\
 & \leq \norm{\hMinn(k)P_n \phi}_{H^{\f12}(\Gamma_n)}
 \norm{P_n  \hat \Pi_n\psi}_{H^{-\f12}(\Gamma_n)} \\
 & \leq C_X(k) \norm{\phi}_{H^{-\f12}(\Gamma)} \norm{\psi}_{H^{-\f12}(\Gamma)}
\end{align*}
where the final inequality holds by the uniform trace inequality Lemma \eqref{lem:trace-uniform}, Proposition \ref{prop:mosco-to-conv} and Lemma \ref{lem:interpolation}.
This shows that the second term on the right hand side of \eqref{eq:final-ineq} is bounded independently of $n$ for each $k \in \C_-$, completing the proof.

\section{Proof of Theorem \ref{th:SCI}}\label{sec:SCI-proof}

The purpose of this section is to prove Theorem \ref{th:SCI} by constructing a sequence of arithmetic algorithms $\Gamma_n: \cS \to \cM$, $n \in \N$, satisfying \eqref{eq:SCI-conv}. 
The construction of the algorithm may be summarised by the following diagram. 
\begin{equation}
	U \xmapsto{\mathrm{pixelate}} \Om_n \xmapsto{\mathrm{triangulate}} \cT_n \xmapsto{\mathrm{Section \, \ref{subsec:num-method}}} g_n \xmapsto{\Gamma_n^{(\mathrm{zeros})}} \Gamma_n(U) 
\end{equation}
We shall approximate $U$ by the pixelation procedure presented in Example \ref{ex:pixelation}, then create a mesh of the inner domain and apply the numerical method presented in Section \ref{subsec:num-method} to compute point values of the analytic function $g_n$. In Section \ref{subsec:compute-zeros} below, we shall construct an arithmetic algorithm $\Gamma_n^{(\mathrm{zeros})}$ capable of compute the zeros of $g_n$ with a-priori error control in  Attouch-Wets distance, from which we obtain the output of $\Gamma_n$. 

\subsection{Applying the Levitin-Marletta method}

Fix $U \in \cS$ and $n \in \N$. 
Let $U_n$ be the pixelation approximation of $U$, as defined by \eqref{eq:pixels-defn}. 
Clearly, $U_n$ depends only on $\mathbbm{1}_U(x_i)$ for a finite number of points $x_i$. 
Next, let $\Om_n := B_X^n \backslash U_n$, where $B_X^n$ is defined as a convex, polygonal subset of $B_X(0)$ whose boundary is obtained by joining $n$ equidistant points on $\partial B_X(0)$ with straight lines. Next, since the corners of $\Om_n$ have angles $\tfrac{\pi}{2} \leq \theta \leq \tfrac{3 \pi}{2}$, 
we may apply standard methods to obtain a shape-regular mesh $\cT_n$. 

As before, let $h_n$ be the largest diameter of an element in $\cT_n$. 
Set parameters $N_n$ and $J_n$ such that Assumption \ref{ass:param} holds. 
The numerical method detailed in Section \ref{subsec:num-method} yields an analytic function $g_n$ such that $g_n(k)$ may be computed in a finite number of arithmetic operations. 
By Theorem \ref{th:main}, we have 
\begin{equation}
	\d_{\mathrm{AW}}(\cZ(g_n),\mathrm{Res}(g_n)) \to 0 \qquad \text{as} \qquad n \to \infty,
\end{equation}
where $\cZ(g_n)$ denotes the zeros of $g_n$ in $\C_-$.

In the next section, we shall construct an arithmetic algorithm $\Gamma_n^{\mathrm{zeros}}$ with access to the point values of $g_n$, such that 
\begin{equation}\label{eq:Gam-zeros-prop}
	\d_{\mathrm{AW}}(\cZ(g_n),\Gamma_n^{\mathrm{zeros}}(g_n)) \leq \frac{1}{n}
\end{equation}
for large enough $n$. 
Setting $\Gamma_n(U) = \Gamma_n^{\mathrm{zeros}}(g_n)$ yields an arithmetic algorithm satisfying \eqref{eq:SCI-conv} as required.

\subsection{Computing the zeros of an analytic function with error control}\label{subsec:compute-zeros}

Let $\cB_n$ be a finite collection of closed boxed $B \subset \C_-$ with non-overlapping interiors such that 
\begin{enumerate}
	\item $\bigcup_{n \in \N}\bigcup_{B \in \cB_n} B = \C_-$,
	\item $\diam(B) \leq 2^{-n}$ for all $B \in \cB_n$,
	\item $\hat{B}_n \subseteq \bigcup_{B \in \cB_n} B $, where $\hat B_n := \set{k \in \C_-: 2^{-n} \leq - \im(k) \leq 2^n, \, |\re(k)| \leq 2^n}$.
\end{enumerate} 
Below, we shall construct an arithmetic algorithm $\Gamma_n^{(\mathrm{dec})}$ such that for any box $B$
we have 
\begin{equation}\label{eq:Gam-dec-prop}
	\Gamma_n^{(\mathrm{dec})}(B, g_n) = \begin{cases}
		\texttt{yes} & \text{if } g_n \text{ does not have any zeros in }B \\
		\texttt{no} & \text{if any }k \in \cZ(g_n) \text{ satisfies  }\dist(k, B) > 2^{-n} 
	\end{cases}.
\end{equation}
In turn, we define
\begin{equation}
	\Gamma_n^{(\mathrm{zeros})}(g_n) = \cup \set{B \in \cB_n: \Gamma_n^{(\mathrm{dec})}(B,g_n) = \texttt{yes}}
\end{equation}

\subsubsection{Error bounds}

By the triangle inequality, we have 
\begin{equation}\label{eq:triangle-AW}
	\d_{\mathrm{AW}}(	\Gamma_n^{(\mathrm{zeros})}(g_n), \cZ(g_n)) \leq \d_{\mathrm{AW}}(	\Gamma_n^{(\mathrm{zeros})}(g_n), \hat{\cZ}_n(g_n)) + \d_{\mathrm{AW}}(	\hat{\cZ}_n(g_n), \cZ(g_n)),
\end{equation}
where 
\begin{equation}
	\hat{\cZ}_n(g_n) := \set{k \in \cZ(g_n) : \dist(k, \hat{B}_n) \leq 2^{-n}}.  
\end{equation}
We estimate the second term on the right hand side of \eqref{eq:triangle-AW} as
\begin{equation}
	 \d_{\mathrm{AW}}(	\hat{\cZ}_n(g_n), \cZ(g_n)) \leq \d_{\mathrm{AW}}(\hat{B}_n,\C_-) \leq \frac{1}{2n}
\end{equation}
for large enough $n$. 
Focusing now on the first term on the right hand side of \eqref{eq:triangle-AW}, notice that the property \eqref{eq:Gam-dec-prop} of $\Gamma_n^{(\mathrm{dec})}$ implies that we have the following. 
\begin{itemize}
	\item
	Let $k \in \hat{Z}_n(g_n)$. Then, there exists $\tilde k \in \hat B_n \cap \cZ(g_n)$ with $|\tilde k - k| \leq 2^{-n}$. Since $\tilde k \in B$ for some $B \in \cB_n$, we also have $\tilde k \in \Gamma_n^{(\mathrm{zeros})}(g_n)$. 
	\item Any $\tilde k \in \Gamma_n^{(\mathrm{zeros})}(g_n)$, must lie in a box $B \in \cB_n$ with $\Gamma_n^{(\mathrm{dec})}(B, g_n) = \texttt{yes}$ so there exists a $\tilde k \in \cZ(g_n)$ with $|k - \tilde k| \leq 2^{-n}$.  In particular, we have $\tilde k \in \hat \cZ_n(g_n)$. 
\end{itemize}
These two properties together clearly imply that the first term on the right hand side of \eqref{eq:triangle-AW} is bounded by $\tfrac{1}{2n}$ for large enough $n$, proving the desired property \eqref{eq:Gam-zeros-prop} of $\Gamma_n^{(\mathrm{zeros})}$. 

\subsubsection{Constructing $\Gamma_n^{(\mathrm{dec})}$}

It remains to construct an arithmetic algorithm $\Gamma_n^{(\mathrm{dec})}$ satisfying \eqref{eq:Gam-dec-prop}. 
Our strategy  revolves around applying the argument principle. 

Fix a closed box $B \subset \C_-$. 
Let $(B_j)_{j \in \N}$ be any sequence of closed boxes such that 
\begin{equation}
 \forall j \in \N : \qquad 	B \subset \cdots \subset B_j \subset B_{j+1} \subset \cdots \subset \C_-,
\end{equation}
\begin{equation}
	B = \bigcap_{j = 1}^\infty B_j, 
\end{equation}
and 
\begin{equation}\label{eq:B_j-dist}
	\inf_{k \in \partial B_j} \dist(k, \partial B) \geq \frac{C_B}{j}. 
\end{equation}

Observe that $g_n$ is the determinant of a matrix with elements that are explicitly expressed in terms of rational and Hankel functions. 
By standard bounds for these functions (see \cite[Chapter 10]{nist} for instance), there exists $\cC_n > 0$, which may be computed in a finite number of arithmetic operations, such that 
\begin{equation}
	\forall j \in \N: \qquad \forall k \in B_j: \qquad |g_n(k)| + |g'_n(k)|+ |g''_n(k)| \leq \cC_n. 
\end{equation}

Next, let $\hat g_{n,j}$ and $\hat g^{(d)}_{n,j}$ be piecewise constant approximations of $g_n$ and $g'_n$
on $\partial B_j$ respectively such that 
\begin{equation}
	\forall\, k \in \partial B_j: \quad \exists\, k_0 \in \partial B_j: \qquad g_n(k_0) =  \hat g_{n,j}(k_0) \qquad \text{and} \qquad |k - k_0| \leq 2^{-j}. 
\end{equation}
By Taylor's theorem, we have on $\partial B_j$,
\begin{equation}\label{eq:g-hat-est-1}
	|g'_n - \hat g^{(d)}_{n,j}| \leq 2^{-j}\cC_n 
\end{equation}
and 
\begin{equation}\label{eq:g-hat-est-2}
	\abs*{\frac{1}{g_n} - \frac{1}{\hat g_{n,j}}} \leq  \frac{2^{-j} \cC_n}{(\inf_{k \in \partial B_j}|g_n|)^2} \leq \frac{2^{-j} \cC_n}{(L_{n,j} - 2^{-j} \cC_n)^2}, 
\end{equation}
where 
\begin{equation}
	L_{n,j} := \inf_{k \in \partial B_j} |\hat g_{n,j}|.
\end{equation}
Notice that $L_{n_j}$ is computable in a finite number of arithmetic operations. 

Next, consider the integral 
\begin{equation}
	\cI_{n,j} := \frac{1}{2\pi i}\oint_{\partial B_j} \frac{\hat g_{n,j}^{(d)}}{\hat g_{n,j}}, 
\end{equation}
which may be computed in a finite number of arithmetic operations. 
It follows from \eqref{eq:g-hat-est-1} and \eqref{eq:g-hat-est-2} that 
\begin{equation}
	\abs*{\frac{1}{2\pi i}\oint_{\partial B_j} \frac{g_n'}{g_n}  - \cI_{n,j} } \leq \frac{1}{2 \pi}|\partial B_j| \br*{\frac{2^{-j} \cC_n}{	L_{n,j} - 2^{-j} \cC_n} + \frac{2^{-j} \cC_n^2}{(L_{n,j} - 2^{-j} \cC_n)^2}} =: D(n,j). 
\end{equation}
Notice that $D(n,j)$ is computable in a finite number of arithmetic operations. 

Since $g_n$ is analytic in $\C_-$, and hence may only have a finite number of zeros in any compact region in $\C_-$, $g_n$ does not have any zeros on $\partial B_j$ for large enough $j$. 
Furthermore, if $g_n$ does not have any zeros on $\partial B$, we must have 
\begin{equation}
L_{n,j} \geq \inf_{k \in \partial B_j} |g_n(k)| \geq C
\end{equation}
for large enough $j$, where $C > 0$ is independent of $j$.
In the other case, that $g_n$ does have at least one zero on $\partial B$, there exists $\nu \in \N$ suhc that the order of those zeros are bounded by $\nu$. 
Then, by property \eqref{eq:B_j-dist} of $\partial B_j$, we have 
\begin{equation}
	L_{n,j} \geq \inf_{k \in \partial B_j} |g_n(k)| \geq \frac{C}{j^\nu}
\end{equation}
for large enough $j$, where $C > 0$ is independent of $j$.
Observe that in either case, we have 
\begin{equation}
	D(n,j) \to 0 \qquad \text{as} \qquad j \to \infty. 
\end{equation}

By performing a finite number of arithmetic operations, we may compute $j$ such that 
\begin{equation}
	D(n,j) < \frac{1}{2}
\end{equation} 
as well as
\begin{equation}
	\inf_{k \in \partial B_j} \dist(k, \partial B) < 2^{-n}. 
\end{equation}
Define an arithmetic algorithm by 
\begin{equation}
	\Gamma_n^{(\mathrm{dec})}(B,g_n) = \begin{cases} \texttt{yes} & \text{if} \quad \cI_{n,j} > \frac{1}{2} \\
	\texttt{no} & \text{otherwise}
	\end{cases}.
\end{equation}
By the argument principle, $\Gamma_n^{(\mathrm{dec})}(B,g_n) = \texttt{yes}$ if and only if $g_n$ has a zero in $\mathrm{int}(B_j)$. 
The desired property \eqref{eq:Gam-dec-prop} holds, completing the proof.

\section{Numerical examples}\label{sec:numerics}

In this section we show numerical results from a MATLAB implementation of our algorithm and assess its performance. We begin with a disk shaped obstacle, for which the resonances can be computed explicitly in terms of zeros of Hankel functions. After that we show results for some domains with fractal boundary.

\FloatBarrier
\subsection{Disk obstacle}\label{sec:disk_numerics}

\begin{figure}[htbp]
	\centering
	\includegraphics[width=0.8\textwidth]{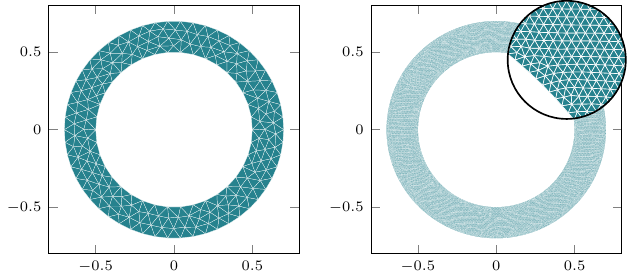}
	\caption{Triangulations of an annulus with an inner radius of 0.5 and outer radius of 0.7 for two different values of the mesh parameter $h$. Left: $h=0.05$, right: $h=0.01$.}
	\label{fig:ball_mesh}
\end{figure}

Consider the obstacle $U = \overline{B_{1/2}(0)}\subset\R^2$, i.e. the disk of radius $\f12$. We chose $X=1$ and used the meshing tool Distmesh \cite{distmesh} to compute triangulations of the annulus $B_X\setminus U$ for the seven values $h\in\{0.08,0.05,0.02,0.01,0.005,0.002,0.001\}$ of the meshing parameter (cf. Figure \ref{fig:ball_mesh} for two examples).

\begin{figure}[htbp]
	\centering
	\includegraphics{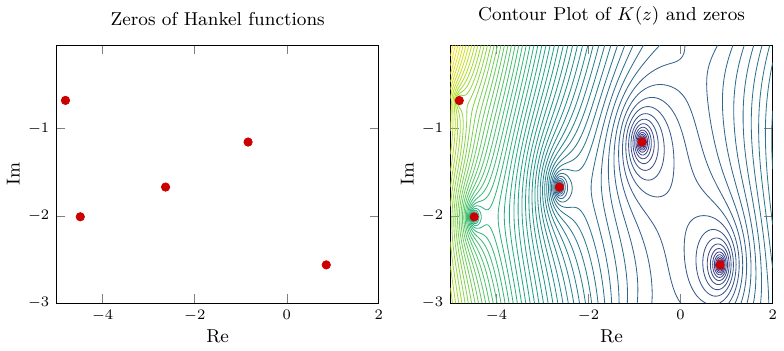}
	\caption{Resonances of $U$ in the complex plane. Left: Computed directly from Hankel functions with high accuracy. Right: Contour plot of $\log|\det(T_n(k))|$, computed as described in Section \ref{}, together with its zeros (red dots). The reference point }
	\label{fig:hankel_resonances}
\end{figure}

Figure \ref{fig:hankel_resonances} (right) shows a contour plot of $T_n(k)$ for $N=6$ and $J=100$. 

\subsubsection{Details of the implementation}\label{sec:optimal_N_heuristic}
Even though the relationship $N h^{\f45}\to 0$ theoretically guarantees convergence, the relative constant between $N$ and $h$ is important and unknown in practice. We therefore use the following heuristic to choose an optimal value of $N$ for any given (finite) $h$:
\begin{enumerate}
	\item Compute the matrix $T_n(k)$ for a large value of $N$,
	\item Consider its diagonal elements (see Figure \ref{fig:K_diag}). By compactness they should tend to 0 at the ends. However, for large $N$ an aliasing-type phenomenon takes over and after reaching a minimum they start growing.
	\item Decrease $N$ until the minimum of $|\diag(T_n(k))|$ is reached at the ends of the matrix.
\end{enumerate}
For the disk obstacle, this process yielded the values in Table \ref{table:optimal_N}.
\begin{table}[h]
	\begin{tabular}{l|lllllll}
		$h$ & 0.08 & 0.05 & 0.02 & 0.01 & 0.005 & 0.002 & 0.001\\
		\hline
		$N$ & 6 & 7 & 10 & 13 & 17 & 28 & 39
	\end{tabular}
	\caption{Optimal values of $N$ for different values of $h$.}
	\label{table:optimal_N}
\end{table}
This relationship is approximately quadratic, i.e. $N\sim h^{-\f12}$, as the right hand plot in Figure \ref{fig:K_diag} shows: plotting $N^2$ against $h^{-1}$ gives an approximately straight line.

\begin{figure}[htbp]
	\centering
	\includegraphics{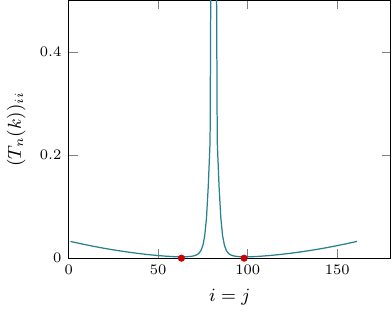}
	\qquad
	\includegraphics{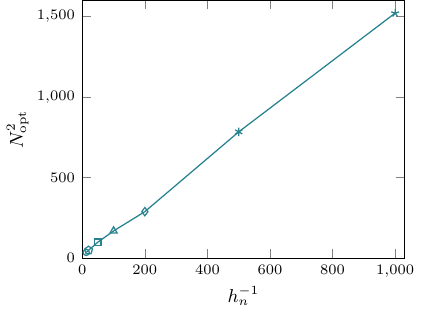}
	\caption{Left: diagonal elements of $|T_n|$ for $h=0.08$, $N=6$. The red dots mark its minima, which we interpret as the optimal size of the matrix $T_n$. Right: square of optimal value of $N$ (determined as in Section \ref{sec:optimal_N_heuristic}), plotted against $h^{-1}$.}
	\label{fig:K_diag}
\end{figure}

\subsubsection{Convergence analysis}\label{sec:convergence_analysis}

In order to test the convergence rate as $h^{-1}$, $N\to\infty$ we chose the resonance near $-0.84 -1.15\i$ (the second from the right in Figure \ref{fig:hankel_resonances}) and increase $h^{-1}$, $N$ according to Table \ref{table:optimal_N}. We shall henceforth refer to the exact value of the resonance as $k_{\text{exact}}$. A zero finding procedure on $H^{(1)}_0$ yields the first 16 digits of $k_{\text{exact}}$ as
\begin{align*}
	k_{\text{exact}} = -0.838549208188362 - 1.154799048234411\i + O(10^{-17}).
\end{align*}
Due to natural limits in memory and computation time, the number $J$ of eigenfunctions in the Aitken's corrector was kept fixed at $J = 100$. In order to ensure the Aitken's error remains negligible nevertheless, we adhered to the following process:
\begin{enumerate}
	\item Choose a reasonable value for $k_0$ by inspecting Figure \ref{fig:hankel_resonances}, say $k_0 = -1-1\i$. Compute a first approximation $\gamma$ of $k_{\text{exact}}$ by minimising $\left|\det(T_n(k))\right|$ (we performed gradient descent until $\left|\det(T_n(k))\right|<10^{-16}$).
	\item Set $k_0 := \gamma$ and recompute the approximation. Call it $\gamma_h$.
	\item Set $k_0 := \gamma_h$, increase $h^{-1}$, $N$ and recompute the approximation $\gamma_h$.
	\item Set $k_0 := \gamma_h$ and proceed in this fashion.
\end{enumerate}
This process ensures that $|k_0 - k|$ remains small for any $k$ that is used in the computation. As a consequence, the Atkinson error remains negligible even for modest values of $J$.

\begin{figure}[htbp]
	\centering
	\includegraphics{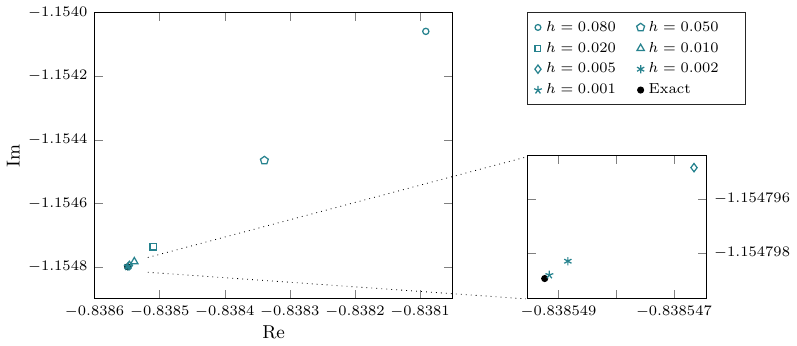}
	\caption{Convergence of successive approximations of $k_{\text{exact}}$ in the complex plane.}
	\label{fig:conv_in_plane}
\end{figure}

The results of this approximation procedure are shown in Figures \ref{fig:conv_in_plane} and \ref{fig:error_plot_ball}. As the plots suggest, the approximation error converges to 0 as $h\to 0$. The slope of the line in Figure \ref{fig:error_plot_ball} suggests a convergence rate of $|k_{\text{exact}} - \gamma_h| \sim h^{2}$, in accordance with the FEM error of a domain with smooth boundary.

\begin{figure}[htbp]
	\centering
	\includegraphics{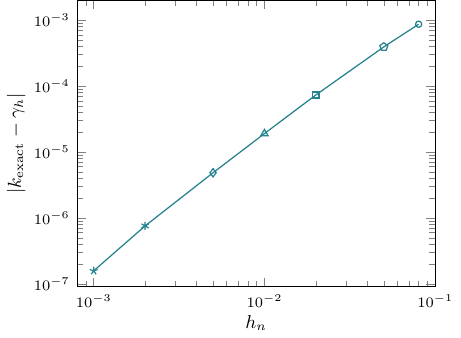}
	\caption{Approximation error $|k_{\text{exact}} - \gamma_h|$ for $h$ in Table \ref{table:optimal_N}.}
	\label{fig:error_plot_ball}
\end{figure}

\FloatBarrier
\subsection{A filled Julia set}\label{sec:julia_numerics}

Next we demonstrate the algorithm's capabilities on domains with fractal boundary. We compute the resonances on a sequence of Julia sets depending on a parameter $q\in[0,0.733]$, which morph from a disk for $q=0$ into an  irregular set (cf. Figure \ref{fig:julia_domains}). 

\begin{figure}[htbp]
	\centering
	\includegraphics[width=0.9\textwidth]{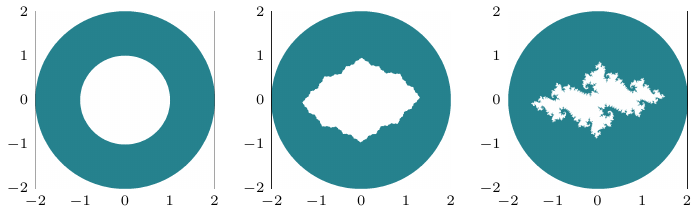}
	\caption{Examples of the Julia sets used in the computation. Left: $q=0$, centre: $q=0.4$, right: $q=0.733$.}
	\label{fig:julia_domains}
\end{figure}

For any complex number $c$ the filled Julia set $\mathcal J_c$ is defined by 
\begin{align}\label{eq:julia_set}
	\mathcal J_c = \{z\in\C\,|\,f^{ \circ k}(z) \text{ bounded as }k\to\infty\},
\end{align}
where $f(z) = z^2+c$ and $f^{ \circ k}$ denote the $k$th iterate of $f$. It can be shown \cite{falconer} that $J_c$ has an interior if and only if $c$ is in the Mandelbrot set. For our numerical experiment we choose
$
c = q(-1+0.2\i), 
$
where $q$ varies from 0 to $0.7$ in steps of $0.05$. If $q=0.75$, then $c$ is outside the Mandlbrot set and $J_c$ fails to satisfy Assumption \ref{ass:U}. In order to capture the behaviour of the resonances at the boundary we added the values $q\in\{0.71, 0.72, 0.73, 0.733\}$ yielding 19 resonance computations in total.

\begin{figure}
	\centering
	\includegraphics[width=0.8\textwidth]{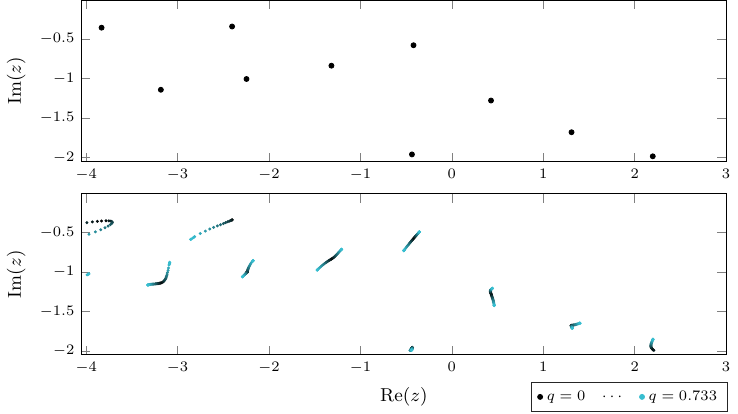}
	\caption{Movement of resonances in the complex plane as $q$ varies from 0 to $0.733$. Top: computed resonances for $q=0$ (the disk obstacle). Bottom: paths traced out by the resonances. Cyan intensity corresponds to larger $q$.}
	\label{fig:julia_seq_in_plane}
\end{figure}

\begin{figure}
	\centering
	\includegraphics[width=0.6\textwidth]{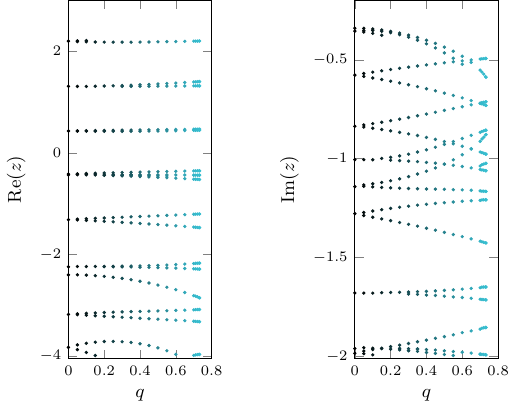}
	\caption{Real and imaginary parts of the computed resonances as a function of $q$. Left: real part, right: imaginary part.}
	\label{fig:julia_res_against_q}
\end{figure}

\begin{remark}[Mesh generation]
	The mesh generation for the filled Julia set was done with a combination of Distmesh (for the outer part) and a pixelation method similar to \cite{spectralpaper}. Pixels were added to the mesh if their midpoint was determined to lie outside the filled Julia set, as determined numerically by truncating the iteration in \eqref{eq:julia_set}.
	The pixel size in Figure \ref{fig:julia_domains} corresponds to the pixel size in our meshing.
\end{remark}

Figures \ref{fig:julia_seq_in_plane} and \ref{fig:julia_res_against_q} show the results of our algorithm for this sequence of sets. For $q=0$ the computation yields the familiar resonances of a disk obstacle (Figure \ref{fig:julia_seq_in_plane} (top) is a scaled version of Figure \ref{fig:hankel_resonances}).
As $q$ increases, the resonances begin to split and drift apart. This is in accordance with the geometry of the associated Julia sets: The disk-shaped domain for $q=0$ is rotationally symmetric. This symmetry is broken more and more as $q$ increases, as Figure \ref{fig:julia_domains} illustrates. As a result the resonances become less and less degenerated.
An animation of the full sequence is available at \href{https://frank-roesler.github.io/images/research/rough_reson_anim.gif}{frank-roesler.github.io/images/research/rough\_reson\_anim.gif}.

\subsection{The Koch Snowflake}\label{sec:koch_numerics}

Finally, we consider the Koch Snowflake, which also satisfies Assumption \ref{ass:U}. A natural sequence $U_n$ with the appropriate convergence properties is given by the Koch prefractals, which can be easily computed and triangulated. In this section we use our algorithm to explore how resonances change as the prefractals approximate the Koch Snowflake.

\begin{figure}[htbp]
	\centering
	\includegraphics[width=0.8\textwidth]{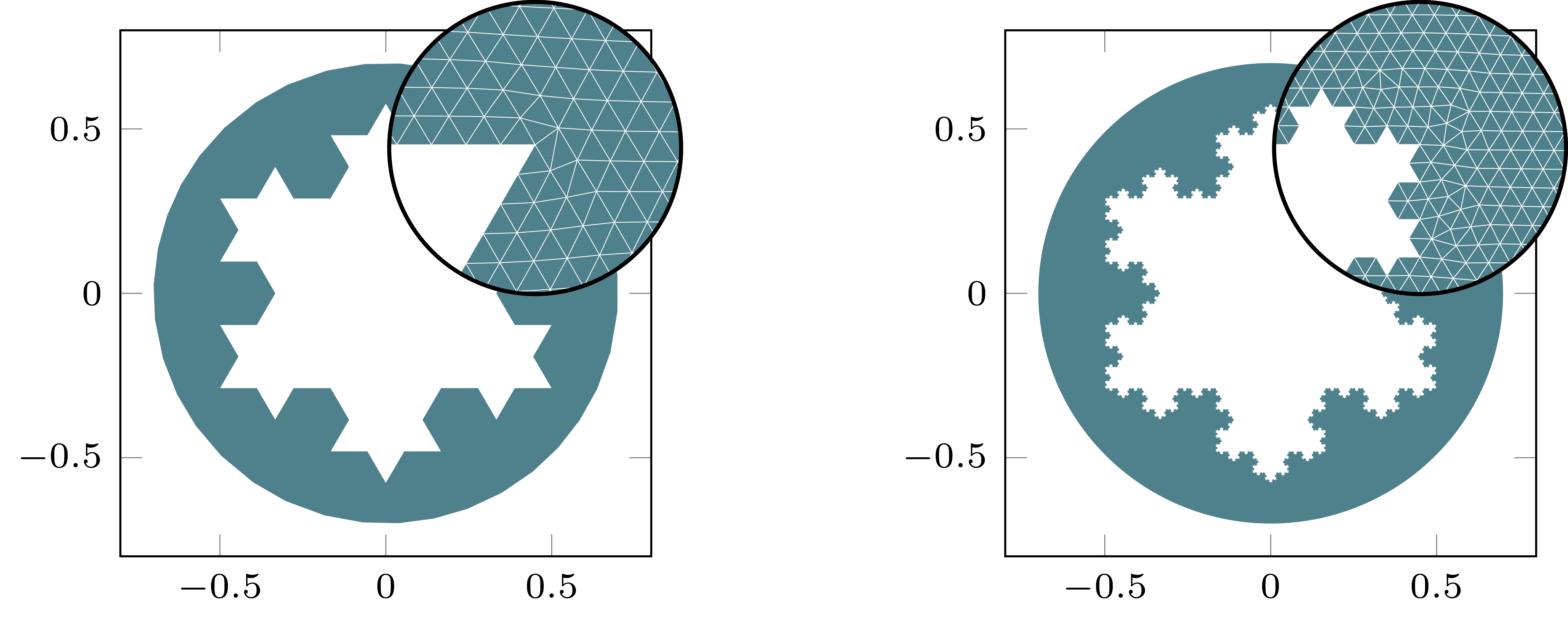}
	\caption{Meshes for both the 2nd (left) and the 5th (right) iteration of the Koch Snowflake used in our computation.}
	\label{fig:koch_meshes}
\end{figure}

We computed triangulated domains for the Koch iterations 2, 3, 4 and 5, cf. Figure \ref{fig:koch_meshes} for illustrations of the 2nd and 5th iterations. Figure \ref{fig:koch_contour_plots} shows example visualisations of the results for the third Koch iteration for two regions in the complex plane (near $-1$ and $-14$,  respectively). The algorithm yields three resonances near $-14$ and one resonance near $-1$.

\begin{figure}[htbp]
	\centering
	\includegraphics[width=0.95\textwidth]{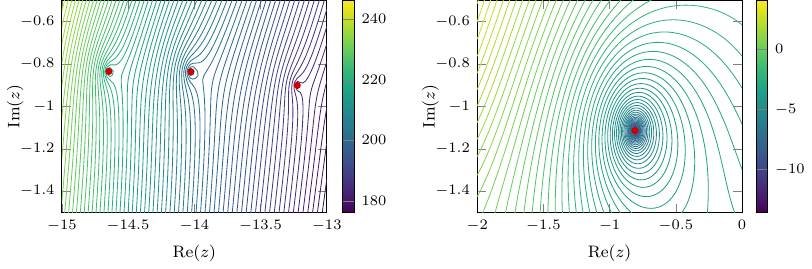}
	\caption{Resonances and contour plots of $\log|\det(T_n(k))|$ for the third iteration of the Koch Snowflake near $k_0=-1-1i$ and $k_0=-14-1i$. }
	\label{fig:koch_contour_plots}
\end{figure}

As the Koch iteration increases and the boundary of the prefractals becomes more and more irregular as they approximate the Koch curve. In this process, more and more small cavities open up in the boundary and one would expect waves of appropriate wave lengths to become increasingly trapped. As a consequence, we would expect the higher resonances (whose wavelength fits the cavities) to depend more strongly on the Koch iteration than the lower ones (whose wave length corresponds to the large scale structure of the domain).

This intuitive understanding is supported by our numerical results, as is shown in Figure \ref{fig:koch_contour_plots}. As the Koch iteration increases from 2 to 5, the three resonances near $-14$ move to the right in steps of order $10^{-1}$. The resonances near $-1$, on the other hand, move by an order of magnitude less, with steps of order $10^{-2}$.

\begin{figure}[htbp]
	\centering
	\includegraphics{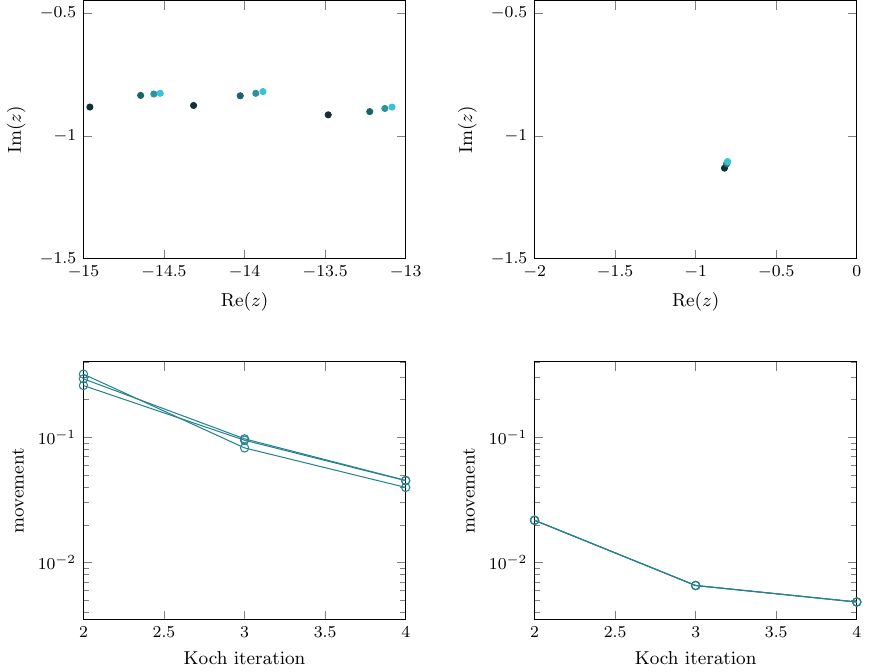}
	\caption{Movement of resonances as the Koch iteration increases. Brighter cyan colours correspond to larger Koch iteration.}
	\label{fig:koch_contour_plots}
\end{figure}

\appendix
\section{Scattering resonances with NtD operators}
\label{app:decomp}

Recall the definition of the solution operator $S(k)$ \eqref{eq:weak-S}. 
Consider the annulus
\begin{equation*}
 A_X  = B_{X}(0) \bs \overline{B_{X - 1}(0)}.
\end{equation*}
Denote the inner boundary of $A_X$ by
\begin{equation*}
 \Gamma_{X-1} := \partial B_{X-1}(0).
\end{equation*}
Let
\begin{equation*}
 S^A \in \cB(H^{-\f12}(\Gamma);H^1(A_X))
\end{equation*}
denote the solution operator for the BVP
\begin{equation}\label{eq:BVPs-SA}
   \begin{cases}
    - \Delta u = 0 \quad \text{on} \quad   A_X \quad  \\
    \del_\nu|_\Gamma u = g, \quad \del_\nu|_{\Gamma_{X-1}} u = - \frac{X}{X-1} g
  \end{cases}.
\end{equation}
Note that the above boundary value problems are well-posed since the compatibility condition
\begin{equation}
  \label{eq:comp-cond}
  \int_{\del A} \del_\nu |_{\del A} u = 0.
\end{equation}

Next, let $R(k):= (-\DND -k^2)^{-1}$, $k \in \C_+$, denote the resolvent operator for the Laplacian $-\DND$ on $L^2(\Om)$ endowed Dirichlet boundary conditions on $\partial U$ and Neumann boundary conditions on $\Gamma$. 
Finally, we introduce a smooth cutoff function $\chi \in C^\infty(A_X)$ such that
\begin{equation*}
 \chi =
 \begin{cases}
 1 \quad \text{near} \quad \Gamma \\
 0 \quad \text{near} \quad \Gamma_{X-1}
 \end{cases}.
\end{equation*}
\begin{lemma}
  \label{lem:S-decomp}
  The operator $ S(k)$ admits the decomposition
  \begin{equation}
    \label{eq:S-decomp}
    S(k)  = \chi S^A + R(k)D_\chi(k) S^A, \qquad k \in \C_+,
  \end{equation}
  where, $D_\chi(k) \in \cB(H^1(A_X);L^2(\Om))$ is a first order differential operator defined by
  \begin{equation}
    \label{eq:D-chi-defn}
    D_\chi(k)  =  k^2 \chi + \Delta(\chi) + 2 \nabla(\chi) \cdot \nabla, \qquad k \in \C_-.
  \end{equation}
\end{lemma}

\begin{proof}
  For arbitrary $g \in H^{-1/2}(\Gamma)$, let $u = S(k)g$ and $v = S^A g$. Consider the function $w := u - \chi v$.
  We aim to show that $w = R(k)D_\chi(k) v$.
  Keeping in mind that $u$ and $v$ are smooth functions by interior elliptic regularity, we compute,
  \begin{align*}
    - \Delta w & = - \Delta u + \Delta(\chi)v + 2 \nabla (\chi) \cdot \nabla(v) + \chi \Delta(v) \\
               & = k^2(u - \chi v ) + k^2 \chi v   + \Delta(\chi) v + 2 \nabla(\chi) \cdot \nabla(v) \\
               & = k^2 w + D_\chi(k) v .
  \end{align*}
  Since $u \in \HND$ and $v \in H^1(A_X)$, an extension by zero shows that $w = u - \chi v \in \HND$, completing the proof.
\end{proof}

\begin{lemma}
\label{lem:SA-decomp}
  There exist a compact operator $K^A$ such that
  \begin{equation*}
    \cN^{1/2}\gamma_\Gamma S^A \cN^{1/2}  = X + K^A.
  \end{equation*}
  Furthermore, for any $s \in \R$, the operators $\cN^sK^A$ and $K^A\cN^s$ are also compact and $S^A:H^{\f12}(\Gamma) \to H^1(A_X)$ is bounded.
\end{lemma}
\begin{proof}
	The result follows from the matrix representation of $S^A$ in the basis $\{e_\alpha\}_{\alpha\in\Z}$, which we will derive by explicit calculation. For $\alpha\in\Z$ let $u = S^A e_\alpha$. A separation ansatz in polar coordinates $u(r,\theta) = \rho(r)\phi(\theta)$ gives the general solution\footnote{We focus on the case $\alpha\neq 0$, which is sufficient for proving compactness.}
	\begin{align}\label{eq:SA-sol-1}
		\rho(r) = c_1 r^\alpha + c_2 r^{-\alpha}
	  \qquad \text{and} \qquad
		\phi(\theta) = e_\alpha(\theta),
	\end{align}
	with $c_1$, $c_2$ to be determined. Imposing the boundary conditions $\del_\nu|_\Gamma u = e_\alpha$, $\del_\nu|_{\Gamma'} u = - \frac{X}{X-1} e_\alpha$ and a direct calculation yields
	\begin{align}\label{eq:SA-sol-2}
		c_1 = \f{X}{\alpha} \f{1}{X^\alpha - (X-1)^\alpha}
      \qquad \text{and} \qquad
		c_2 = \f{X}{\alpha} \f{1}{(X-1)^{-\alpha} - X^{-\alpha}}
	\end{align}
	and thus
	\begin{align}\label{eq:radial_part_of_S^A}
		\rho(r) = \f{X}{\alpha} \left( \f{r^\alpha}{X^\alpha - (X-1)^\alpha} + \f{r^{-\alpha}}{(X-1)^{-\alpha} - X^{-\alpha}} \right)
	\end{align}
	From \eqref{eq:radial_part_of_S^A} we immediately conclude that
	\begin{align*}
		\gamma_\Gamma S^A(e_\alpha)
		= \f{X}{\alpha} \left(
			\f{X^\alpha + (X-1)^\alpha}{X^\alpha - (X-1)^\alpha}
		\right) e_\alpha(\theta).
	\end{align*}
	Hence $\gamma_\Gamma S^A$ is diagonal in the basis $\{e_\alpha\}_{\alpha\in\Z}$. We conclude that for any $\alpha\neq 0$
	\begin{align}\label{eq:NS^AN_decay}
		\bigl( \cN^{\f12}\gamma_\Gamma S^A\cN^{\f12} \bigr)_{\alpha\alpha}
		&= X\f{|\alpha|}{\alpha} \f{X^\alpha + (X-1)^\alpha}{X^\alpha - (X-1)^\alpha}
		\\
		&= \begin{cases}
			X + 2X\left( \bigl(\f{X}{X-1}\bigr)^\alpha - 1 \right)^{-1} & \text{for }\alpha>0
			\\[2mm]
			X - 2X\left( 1 - \left(\f{X-1}{X}\right)^\alpha \right)^{-1} & \text{for }\alpha<0
		\end{cases}
	\end{align}
	Since the terms $((X/(X-1))^\alpha - 1)^{-1}$ and $((X-1)/X)^\alpha -1)^{-1}$ decay exponentially as $\alpha\to +\infty, -\infty$, respectively, eq. \eqref{eq:NS^AN_decay} immediately implies the assertion.
	
	To prove boundedness of $S^A:H^{\f12}(\Gamma) \to H^1(A_X)$, consider $u_\alpha := S^A(\alpha^{-\f12} e_\alpha)$. The above calculation yields
	\begin{align*}
		u_\alpha(r,\theta) = \alpha^{-\f12}\rho(r)e_\alpha(\theta).
	\end{align*}
	A lengthy calculation yields explicit formulas for $\|u_\alpha\|_{L^2}$, $\|u_\alpha'\|_{L^2}$, which imply
	\begin{align}
		\|u_\alpha\|_{L^2} &\sim \alpha^{-1} \quad\text{ as }\alpha\to\pm\infty
		\\
		\|u_\alpha'\|_{L^2} &\sim 1 \quad\text{ as }\alpha\to\pm\infty,
	\end{align}
	which immediately implies the desired $H^1$-boundedness.
\end{proof}

\begin{lemma}
\label{lem:K-prop}
 There exists an analytic family of compact operators $K(k)$, $k \in \C_-$,
 on $L^2(\Gamma)$ such that
 \begin{equation*}
  \frac{1}{2}\cN^{\f12} (N_1(k)+ \Min(k) N_2(k)) \cN^{-\f12} = I + K(k), \qquad k \in \C_-.
 \end{equation*}
Furthermore, $K(k)$ is bounded on $L^2(\Gamma)$ locally uniformly for $k \in \C_-$.
\end{lemma}
\begin{proof}
	Let $k\in\C_-$. Recalling the definitions of $N_1$ and $N_2$ we have
	\begin{equation*}
 		N_1(k)_{\alpha\alpha} := \frac{H^{(1)}_{|\alpha|}(kX)}{A_{|\alpha|}(k)}, \qquad 
 		N_2(k)_{\alpha\alpha} := \frac{(H_{|\alpha|}^{(1)})'(kX)}{A_{|\alpha|}(k)}
	\end{equation*}
	and $N_i(k)_{\alpha\beta}=0$ otherwise, where $A_\nu := -i \sqrt{\frac{2}{\pi \nu}} \br*{\frac{e k X}{2 \nu}}^{-\nu}$.
	By well-known properties for Bessel functions (cf. \eqref{eq:Hankel-asym}) we have
	\begin{align}\label{eq:N1_asymptoticFormula}
		\lim_{|\alpha|\to\infty} N_1(k)_{\alpha\alpha} &= 1
	\end{align}
	and hence
	\begin{align}
		N_1(k) = I + K_1(k),
	\end{align}
	where $K_1(k) = \diag\left(\f{H^{(1)}_{|\alpha|}(kX)}{A_{|\alpha|}(k)} - 1, \alpha\in\Z\right)$ is compact. Similarly, for $N_2$ we have
	\begin{align}\label{eq:LN2_besselFormula}
		N_2(k)_{\alpha\alpha} &= A_{|\alpha|}(k)^{-1}\bigg( k H_{|\alpha|-1}^{(1)}(kX) - \f{|\alpha|}{X} H_{|\alpha|}^{(1)}(kX) \bigg),
	\end{align}
	where we have used the general formula $\f{d}{dz}H^{(1)}_\nu(z) = H_{\nu-1}^{(1)}(z) - \f{\nu}{z}H_{\nu}^{(1)}(z)$. The first term on the right-hand side of \eqref{eq:LN2_besselFormula} behaves like 
	\begin{align*}
		 \frac{k H_{|\alpha|-1}^{(1)}(kX)}{A_{|\alpha|}(k)} &\sim c\f{(|\alpha|-1)^{|\alpha|-1}}{|\alpha|^{|\alpha|}}
		 \\
		 &\xrightarrow{|\alpha|\to\infty} 0
	\end{align*}
	for a suitable constant $c$. Hence, comparing to \eqref{eq:LN2_besselFormula} we obtain
	\begin{align}\label{eq:LN2_asymptoticFormula}
		N_2(k) &= -X^{-1}\cN + K_2(k),
	\end{align}
	where $K_2(k) = \diag\left(\frac{k H_{|\alpha|-1}^{(1)}(kX)}{A_{|\alpha|}(k)} , \alpha\in\Z\right)$ is compact. Collecting results, we have shown
	\begin{align}
		\cN^{\f12}N_1(k)\cN^{-\f12} &= I + K_1(k)
		\label{eq:N1_finalFormula}
		\\
		N_2(k)\cN^{-\f12} &= -X^{-1}\cN^{\f12} + \cN^{-\f12}K_2(k).
		\label{eq:N2_finalFormula}
	\end{align}
	It remains to control $\Min(k)$. Combining Lemmas \ref{lem:S-decomp} and \ref{lem:SA-decomp} we have
	\begin{align*}
		\Min(k) &= \gamma_\Gamma S(k)
		\\
		&= \gamma_\Gamma (\chi S^A + R(k)D_\chi(k)S^A)
	\end{align*}
	and 
	\begin{align*}
		\cN^{\f12}\Min(k)\cN^{\f12} &= \cN^{\f12} \gamma_\Gamma \chi S^A \cN^{\f12} + \cN^{\f12} \gamma_\Gamma R(k)D_\chi(k)S^A \cN^{\f12}
		\\
		&= \cN^{\f12} \gamma_\Gamma S^A \cN^{\f12} + \cN^{\f12} \gamma_\Gamma R(k)D_\chi(k)S^A \cN^{\f12}
		\\
		&= X+K^A + \cN^{\f12} \gamma_\Gamma R(k)D_\chi(k)S^A \cN^{\f12}.
	\end{align*}
	We simplify notation by writing
	\begin{align}\label{eq:Min_decompFormula}
		\cN^{\f12}\Min(k)\cN^{\f12} = X+K_3(k),
	\end{align}
	where $K_3(k) = K^A + \cN^{\f12} \gamma_\Gamma R(k)D_\chi(k)S^A \cN^{\f12}$. Combining \eqref{eq:Min_decompFormula} and \eqref{eq:N2_finalFormula} we obtain
	\begin{align}
		\cN^{\f12}\Min(k)N_2(k)\cN^{-\f12} &= \cN^{\f12}\Min(k)(-X^{-1}\cN^{\f12} + \cN^{-\f12}K_2(k))
		\nonumber
		\\
		&= -X^{-1}\cN^{\f12}\Min(k)\cN^{\f12} + \cN^{\f12}\Min(k)\cN^{-\f12}K_2(k)
		\nonumber
		\\
		&= \cN^{\f12}\Min(k)\cN^{\f12}(-X^{-1} + \cN^{-1}K_2(k))
		\nonumber
		\\
		&= (X+K_3(k))(-X^{-1} + \cN^{-1}K_2(k))
		\nonumber
		\\
		&= I - X^{-1}K_3(k) + X\cN^{-1}K_2(k) + K_3(k)\cN^{-1}K_2(k)
		\label{eq:MinN2_finalFormula}
	\end{align}
	Combining \eqref{eq:MinN2_finalFormula} and \eqref{eq:N1_finalFormula} we obtain the final formula
	\begin{align}
		\frac{1}{2}\cN^{\f12} (N_1(k)+ \Min(k) N_2(k)) \cN^{-\f12}
		&= I + K(k),
	\end{align}
	where
	\begin{align}\label{eq:K(k)_def}
		2K(k) = K_1(k) - X^{-1}K_3(k) + X\cN^{-1}K_2(k) + K_3(k)\cN^{-1}K_2(k).
	\end{align}
	To prove the assertion, it remains to show that $K(k)$ is compact for every $k\in \C_-$ and that $K(k)$ is locally uniformly bounded on $L^2(\Gamma)$. 
	
	Local uniform boundedness follows from continuity of $H_\nu^{(1)}$, $D_\chi$ in $\C_-$ and the fact that $A_{\nu}$ has no zeros in $\C_-$. To prove compactness we consider each term in \eqref{eq:K(k)_def} separately. Compactness of $K_1$ is already established, thus we focus on $K_3$. $K^A$ is compact by Lemma \ref{lem:SA-decomp}, so it suffices to prove compactness of $\cN^{\f12} \gamma_\Gamma R(k)D_\chi(k)S^A \cN^{\f12}$. Employing Lemmas \ref{lem:S-decomp}, \ref{lem:SA-decomp} we have the following sequence of bounded operators
	\begin{align*}
		L^2(\Gamma) \xrightarrow{\cN^{\f12}} 
		H^{\f12}(\Gamma) \xrightarrow{S^A}
		H^1(A_X) \xrightarrow{D_\chi(k)}
		L^2(\Omega) \xrightarrow{\chi R(k)}
		H^2(\Omega) \xrightarrow{\gamma_\Gamma}
		H^{\f32}(\Gamma) \xrightarrow{\cN^{\f12}}
		H^1(\Gamma),
	\end{align*}
	where we have used the fact that $\gamma_\Gamma\chi R(k) = \gamma_\Gamma R(k)$. We conclude that the range of $\cN^{\f12} \gamma_\Gamma R(k)D_\chi(k)S^A \cN^{\f12}$ is compactly embedded in $L^2(\Gamma)$, proving compactness of the operator and of $K_3(k)$. 
	
	It only remains to prove compactness of $\cN^{-1}K_2(k)$. However, this is trivial, since $K_2(k)$ is bounded and $\cN^{-1}$ is compact.
\end{proof}

\bibliography{mybib}

\begin{thebibliography}{10}

\bibitem{ammari2009layer}
H.~Ammari, H.~Kang, and H.~Lee.
\newblock {\em Layer potential techniques in spectral analysis}.
\newblock Number 153. American Mathematical Soc., 2009.

\bibitem{sci3}
J.~Ben-Artzi, M.~J. Colbrook, A.~C. Hansen, O.~Nevanlinna, and M.~Seidel.
\newblock {Computing Spectra--On the Solvability Complexity Index Hierarchy and
  Towers of Algorithms}.
\newblock {\em arXiv:1508.03280}, 2020.

\bibitem{sci2}
J.~Ben-Artzi, A.~C. Hansen, O.~Nevanlinna, and M.~Seidel.
\newblock New barriers in complexity theory: on the solvability complexity
  index and the towers of algorithms.
\newblock {\em C. R. Math. Acad. Sci. Paris}, 353(10):931--936, 2015.

\bibitem{seashell}
J.~Ben-Artzi, M.~Marletta, and F.~R\"{o}sler.
\newblock Computing the sound of the sea in a seashell.
\newblock {\em Found. Comput. Math.}, 22(3):697--731, 2022.

\bibitem{JEMS}
J.~Ben-Artzi, M.~Marletta, and F.~R\"{o}sler.
\newblock Computing scattering resonances.
\newblock {\em J. Eur. Math. Soc. (JEMS)}, 25(9):3633--3663, 2023.

\bibitem{BrezisInterpolation}
H.~Brezis and P.~Mironescu.
\newblock Gagliardo-{N}irenberg inequalities and non-inequalities: the full
  story.
\newblock {\em Ann. Inst. H. Poincar\'{e} C Anal. Non Lin\'{e}aire},
  35(5):1355--1376, 2018.

\bibitem{CW21}
S.~N. Chandler-Wilde, D.~P. Hewett, A.~Moiola, and J.~Besson.
\newblock Boundary element methods for acoustic scattering by fractal screens.
\newblock {\em Numer. Math.}, 147(4):785--837, 2021.

\bibitem{sci5}
M.~J. Colbrook.
\newblock {On the computation of geometric features of spectra of linear
  operators on Hilbert spaces}.
\newblock {\em Foundations of Computational Mathematics}, 2019.

\bibitem{sci4}
M.~J. Colbrook and A.~C. Hansen.
\newblock The foundations of spectral computations via the solvability
  complexity index hierarchy.
\newblock {\em J. Eur. Math. Soc. (JEMS)}, 25(12):4639--4718, 2023.

\bibitem{nist}
{\it NIST Digital Library of Mathematical Functions}.
\newblock \url{https://dlmf.nist.gov/}, Release 1.1.12 of 2023-12-15.
\newblock F.~W.~J. Olver, A.~B. {Olde Daalhuis}, D.~W. Lozier, B.~I. Schneider,
  R.~F. Boisvert, C.~W. Clark, B.~R. Miller, B.~V. Saunders, H.~S. Cohl, and
  M.~A. McClain, eds.

\bibitem{resonances}
S.~Dyatlov and M.~Zworski.
\newblock {\em Mathematical theory of scattering resonances}, volume 200 of
  {\em Graduate Studies in Mathematics}.
\newblock American Mathematical Society, Providence, RI, 2019.

\bibitem{falconer}
K.~Falconer.
\newblock {\em Fractal {{Geometry}}: {{Mathematical Foundations}} and
  {{Applications}}}.
\newblock {John Wiley \& Sons}, 2004.

\bibitem{asym1}
J.~{Fleckinger} and D.~G. Vassiliev.
\newblock An {{Example}} of a {{Two}}-{{Term Asymptotics}} for the {{``Counting
  Function"}} of a {{Fractal Drum}}.
\newblock {\em Transactions of the American Mathematical Society},
  337(1):99--116, 1993.

\bibitem{absorb}
S.~Félix, B.~Sapoval, M.~Filoche, and M.~Asch.
\newblock Enhanced wave absorption through irregular interfaces.
\newblock {\em Europhysics Letters}, 85(1):14003, jan 2009.

\bibitem{sci1}
A.~C. Hansen.
\newblock On the {{Solvability Complexity Index}}, the n-pseudospectrum and
  approximations of spectra of operators.
\newblock {\em Journal of the American Mathematical Society}, 24(1):81--81,
  2011.

\bibitem{shape}
M.~Hinz, A.~Rozanova-Pierrat, and A.~Teplyaev.
\newblock Non-{L}ipschitz uniform domain shape optimization in linear
  acoustics.
\newblock {\em SIAM Journal on Control and Optimization}, 59(2):1007--1032,
  2021.

\bibitem{asym2}
C.~Hua and B.~D. Sleeman.
\newblock Fractal drums and then-dimensional modified {{Weyl}}-{{Berry}}
  conjecture.
\newblock {\em Communications in Mathematical Physics}, 168(3):581--607, 1995.

\bibitem{asym3}
M.~L. Lapidus.
\newblock Fractal drum, inverse spectral problems for elliptic operators and a
  partial resolution of the {{Weyl}}-{{Berry}} conjecture.
\newblock {\em Transactions of the American Mathematical Society},
  325(2):465--529, 1991.

\bibitem{asym5}
M.~L. Lapidus and C.~Pomerance.
\newblock Counterexamples to the modified {W}eyl-{B}erry conjecture on fractal
  drums.
\newblock {\em Math. Proc. Cambridge Philos. Soc.}, 119(1):167--178, 1996.

\bibitem{LevitinMarletta}
M.~Levitin and M.~Marletta.
\newblock A simple method of calculating eigenvalues and resonances in domains
  with infinite regular ends.
\newblock {\em Proc. Roy. Soc. Edinburgh Sect. A}, 138(5):1043--1065, 2008.

\bibitem{hyperbolic}
M.~Levitin and A.~Strohmaier.
\newblock Computations of eigenvalues and resonances on perturbed hyperbolic
  surfaces with cusps.
\newblock {\em Int. Math. Res. Not. IMRN}, (6):4003--4050, 2021.

\bibitem{asym4}
M.~Levitin and D.~Vassiliev.
\newblock Spectral {{Asymptotics}}, {{Renewal Theorem}}, and the {{Berry
  Conjecture}} for a {{Class}} of {{Fractals}}.
\newblock {\em Proceedings of the London Mathematical Society},
  s3-72(1):188--214, 1996.

\bibitem{mclean2000strongly}
W.~C.~H. McLean.
\newblock {\em Strongly elliptic systems and boundary integral equations}.
\newblock Cambridge university press, 2000.

\bibitem{distmesh}
P.-O. Persson and G.~Strang.
\newblock A simple mesh generator in matlab.
\newblock {\em SIAM review}, 46(2):329--345, 2004.

\bibitem{roddick}
G.~Roddick.
\newblock Computation of scattering matrices and their derivatives for
  waveguides.
\newblock {\em J. Comput. Appl. Math.}, 396:Paper No. 113453, 24, 2021.

\bibitem{spectralpaper}
F.~R\"{o}sler and A.~Stepanenko.
\newblock Computing eigenvalues of the {L}aplacian on rough domains.
\newblock {\em Math. Comp.}, 93(345):111--161, 2023.

\bibitem{taylor-lay}
A.~E. Taylor and D.~C. Lay.
\newblock {\em Introduction to functional analysis}, volume~1.
\newblock Wiley New York, 1958.

\end{thebibliography}
\bibliographystyle{abbrv}

\end{document}